\documentclass[reqno,a4paper]{amsart}
\usepackage{amsfonts,amsmath,times,amssymb}
\usepackage{a4wide}
\usepackage{tikz}

\usepackage[sc]{mathpazo} 
\usepackage[scaled]{helvet} 
\usepackage{eulervm} 

\definecolor{grau}{rgb}{0.65,0.65,0.65}
\definecolor{dblau}{rgb}{0,0,0.45}
\definecolor{blau}{rgb}{0,0,0.75} 
\usepackage[pdftex]{hyperref}
\hypersetup{colorlinks,linkcolor=blau,citecolor=blue,urlcolor=blau}

\usepackage{algorithm}
\usepackage{algpseudocode}

\algnewcommand\algorithmicinput{\textbf{Input: }}
\algnewcommand\Input[1]{\item[\algorithmicinput] #1}

\algnewcommand\Alg{\item \textbf{Algorithm: }}

\algnewcommand\algorithmicoutput{\textbf{Output: }}
\algnewcommand\Output[1]{\item[\algorithmicoutput] #1}

\algnewcommand\algorithmicresult{\textbf{Result: }}\algnewcommand\Result[1]{\item[\algorithmicresult] #1}

\algdef{SE}[DOWHILE]{Do}{doWhile}{\algorithmicdo}[1]{\algorithmicwhile\ #1}%

\definecolor{grau}{rgb}{0.65,0.65,0.65}
\definecolor{dblau}{rgb}{0,0,0.45}
\definecolor{blau}{rgb}{0,0,0.75} 
\usepackage[pdftex]{hyperref}
\hypersetup{colorlinks,linkcolor=blau,citecolor=blue,urlcolor=blau}

\newcommand{\myt}[1]{{\it\color{dblau}#1}}

\newcommand{\dit}{\ensuremath{(b,d)}\text{-ary ITs}}
\newcommand{\port}{\ensuremath{(b,\alpha)}\text{-PORTs}}
\newcommand{\Dit}{\ensuremath{(b,d)}\text{-ary} \text{in\-crea\-sing} \text{trees}}
\newcommand{\Port}{\ensuremath{(b,\alpha)}\text{-plane} \text{oriented} \text{recursive} \text{trees}}
\newcommand{\Rec}{\text{bucket recursive trees}}

\allowdisplaybreaks

\newtheorem{theorem}{Theorem}
\newtheorem{lemma}{Lemma}
\newtheorem{prop}{Proposition}
\newtheorem{coroll}{Corollary}
\theoremstyle{definition}
\newtheorem{remark}{Remark}
\newtheorem{example}{Example}
\newtheorem{defi}{Definition}
\newtheorem{urn}{Urn}

\def\P{{\mathbb {P}}}
\def\E{{\mathbb {E}}}

\def\N{{\mathbb {N}}}
\def\R{{\mathbb {R}}}

\newcommand{\bN}{\ensuremath{\mathbf{N}}}

\newcommand{\bQ}{\ensuremath{\mathbf{Q}}}

\DeclareMathOperator{\Be}{Be}

\DeclareMathOperator{\NegBin}{NegBin}

\DeclareMathOperator{\grad}{deg}

\newcommand{\ind}{\ensuremath{\mathbb{I}}}

\newcommand{\unord}{\ensuremath{^{[U]}}}

\newcommand{\ptree}{\ensuremath{^{[\mathcal{T}]}}}
\newcommand{\pgrow}{\ensuremath{^{[e]}}}

\newcommand{\law}{\ensuremath{\stackrel{(d)}=}}
\newcommand{\claw}{\ensuremath{\xrightarrow{(d)}}}

\newcommand{\fallfak}[2]{\ensuremath{#1^{\underline{#2}}}}
\newcommand{\auffak}[2]{\ensuremath{#1^{\overline{#2}}}}

\author[M.~Kuba]{Markus Kuba}
\address{Markus Kuba\\
Department Applied Mathematics and Physics\\
University of Applied Sciences - Technikum Wien\\
H\"ochst\"adtplatz 5, 1200 Wien} %
\email{kuba@technikum-wien.at}

\author[A.~Panholzer]{Alois Panholzer}
\address{Alois Panholzer\\
Institut f{\"u}r Diskrete Mathematik und Geometrie\\
Technische Universit\"at Wien\\
Wiedner Hauptstr. 8-10/104\\
1040 Wien, Austria} \email{Alois.Panholzer@tuwien.ac.at}

\date{\today}

\title{On bucket increasing trees, clustered increasing trees and increasing diamonds}

\begin{document}

\begin{abstract}
In this work we analyze bucket increasing tree families. 
We introduce two simple stochastic growth processes, generating random bucket increasing trees of size $n$, complementing 
the earlier result of Mahmoud and Smythe~\cite{MahSmy1995} for bucket recursive trees. On the combinatorial side, 
we define multilabelled generalizations of the tree families $d$-ary increasing trees and generalized plane-oriented recursive trees. 
Additionally, we introduce a clustering process for ordinary increasing trees and relate it to bucket increasing trees. 
We discuss in detail the bucket size two and present a bijection between such bucket increasing tree families and certain families of graphs called increasing diamonds, providing an explanation for phenomena observed by Bodini et al.~\cite{Hwang2016}. 

\smallskip

Concerning structural properties of bucket increasing trees, we analyze the tree parameter $K_n$. 
It counts the initial bucket size of the node containing label $n$ in a tree of size $n$ and is closely related to the distribution of node types. 
Additionally, we analyze the parameters descendants of label $j$ and degree of the bucket containing label $j$, providing distributional decompositions, 
complementing and extending earlier results~\cite{BucketPanKu2009}.

\end{abstract}

\maketitle

\section{Introduction}
\myt{Increasing trees} or \myt{increasingly labelled trees} are rooted labelled trees. 
The nodes of a tree $T$ of size $|T| = n$ are labelled with distinct integers from a label set $\mathcal{M}$ of size $|\mathcal{M}|=n$. 
Here, the size $|T|$ of a tree denotes the number of vertices of $T$ (and thus coincides with the number of labels). One chooses as label set the first $n$ positive integers, i.e., $\mathcal{M} = [n] := \{1, 2, \dots, n\}$, in such a way that the label of any node in the tree is smaller than the labels of its children. As a consequence, the labels of each path from the root to an arbitrary node in the tree are forming an increasing sequence, which explains the name of such a labelling. Various increasing tree models turned out to be appropriate in order to describe the growth behavior of quantities in a number of applications and occurred in the probabilistic literature.  E.g., they are used to describe the spread of epidemics, to model pyramid schemes, and as a simplified growth model of the world wide web. See Mahmoud and Smythe~\cite{MahSmy1995-2} for a survey collecting results about recursive trees, a subfamily of increasing trees, prior 1995. For recent results about increasing trees we refer to the books of Drmota~\cite{Drmota}, Flajolet and Sedgewick~\cite{FlaSed} and references therein.

\smallskip

In above definition of increasing trees each node in the tree gets exactly one label. 
Here in this work we discuss an extensions of increasing trees to \myt{bucket increasing trees}. These are \myt{multilabelled increasing tree families}. All the nodes $v$ in the tree $T$ are \myt{buckets} with a maximal capacity of $b$ labels. Here and throughout this work the integer $b\in\N$ denotes the maximal capacity or bucket size. The integer $c=c(v)$ denotes
the current capacity or load of a node $v\in T$, with $1 \le c \le b$.
We always call a node $v$ with capacity $c(v) = b$ \myt{saturated} and otherwise \myt{unsaturated}. We impose the additional restriction that all internal nodes (i.e., non-leaves) in the tree must be saturated. In contrast, the leaves might be either saturated or unsaturated. The size of a tree $|T|$ is here and throughout this work measured by the sum of all node capacities $c(v)$: $|T| = \sum_{v\in T} c(v)$; equivalently, it is given by the number of labels in the tree $T$. We note in passing that tree families where all leaves have to be saturated have been considered and discussed recently in~\cite{KubPan2012,KubPan2016} in connection with so-called hook-length formulas. Moreover, closely related combinatorial objects named increasing diamonds have been studied by Bodini et al.~\cite{Hwang2016}, as discussed later.

\smallskip 

A specific family of bucket increasing trees called \myt{bucket recursive trees} has been introduced by Mahmoud and Smythe~\cite{MahSmy1995}. The probabilistic description of random bucket recursive trees is given by a \myt{stochastic growth rule}: a tree grows by progressive attraction of increasing integer labels. Such growth rules are part of a general (preferential) attachment rule with fertility and aging, see~Berger et al.~\cite{Borgs2004,Borgs2005}. A general combinatorial model for bucket recursive trees has been introduced in~\cite{BucketPanKu2009}.

\smallskip

The aims of this paper are the following: first, we discuss the general combinatorial framework for bucket increasing trees 
and the corresponding random tree model. We introduce and analyze two new families of bucket increasing trees.
For the two new families of \myt{$(b,d)$-ary increasing trees} and \myt{$(b,\alpha)$-plane oriented recursive trees} we present both, a stochastic growth rule and a combinatorial description. Then, we prove that both descriptions are equivalent. As a byproduct of our study we complement previous results~\cite{MahSmy1995,BucketPanKu2009} for bucket recursive trees. Furthermore we present a clustering map for ordinary increasing trees, which maps them to bucket increasing trees. In the special case of bucket size $b=2$ we state a bijection
between bucket increasing trees and so-called increasing diamonds studied by Bodini et al.~\cite{Hwang2016}, providing an explanation for phenomena
observed in~\cite{Hwang2016}.

\smallskip

We obtain the combinatorial description using a generalization of a class of weighted tree families, so called simple families of increasing trees, to a class of bucket trees, which we call families of bucket increasing trees. 
All considered tree families will then be special instances of a bucket increasing tree family. As in the previous analysis of bucket recursive trees the gain of the combinatorial description provided here is that the natural combinatorial
decomposition of a bucket increasing tree into a root bucket and its subtrees will lead to a recursive description of several important
tree parameters in random bucket recursive trees. This combinatorial decomposition can be translated into certain differential equations for suitably defined generating functions. On the other hand the stochastic growth rules will allow us to present decompositions of random variables of interest. 

\smallskip

The combination of both methods -- a combinatorial approach leading to exact expressions and
the stochastic growth rule leading to decompositions -- turns out to be a particularly useful tool for a
variety of parameters. First we study the random variables (r.v.\ for short) $K_{n}$ counting the size of the bucket containing label $n$. 
Explicit results for the probability mass function of $K_n$ are obtained by the combinatorial approach.
The explicit expression for $\P\{K_n=m\}$ readily leads to a discrete limit law for $K_n$ as $n$ tends to infinity. 
Then, we relate $K_n$ to the random variables $N_{n,k}$, counting the number of buckets of capacity $k$, $1\le k\le b$. 

\smallskip

Next, we turn to the analysis of the r.v.\ $Y_{n,j}$, counting the number of descendants of label $j$. 
The results for $K_n$ and its limit law are used to obtain the exact distribution of $Y_{n,j}$,
as well as several limit laws depending on the growth of $j=j(n)$ as $n$ tends to infinity. Moreover, we study the random variable $X_{n,j}$ counting the out-degree of the bucket containing label $j$ in a size $n$ bucket tree. The analysis of $X_{n,j}$ is based on a stopping time closely related to $Y_{n,j}$.
We also provide decompositions of the random variables $X_{n,j}$ and $Y_{n,j}$ in terms of $K_n$ for fixed $n$.

\smallskip

\subsection{Notation} 
We denote with $X \stackrel{(d)}{=} Y$  the equality in distribution of two random variables $X$ and $Y$. We write
$X_{n} \xrightarrow{(d)} X$ for the weak convergence (i.e., convergence in distribution) of a sequence
of random variables $X_{n}$ to a r.v.\ $X$. Let $H_{n} := \sum_{k=1}^{n} \frac{1}{k}$ denote the harmonic numbers and 
$H_{n+\alpha}-H_{\alpha} := \sum_{k=1}^{n} \frac{1}{k+\alpha}$ the continuation of the harmonic numbers for
a complex $\alpha \in \mathbb{C}\setminus\{-1, -2, -3, \dots\}$. 
Here and throughout this work we use the notation $\fallfak{x}{s}:=x(x-1)\dots(x-(s-1))$ for the falling factorials, and $\auffak{x}{s}:=x(x+1)\dots(x+s-1)$ for the rising factorials, $s\in\N_0$.\footnote{The notation $\fallfak{x}{s}$ and $\auffak{x}s$ was introduced and popularized by Knuth; alternative notations for the falling factorials 
include the Pochhammer symbol $(x)_s$, which is unfortunately sometimes also used for the rising factorials.}
Throughout this work we will use the abbreviation \dit\ for $(b,d)$-ary bucket increasing trees and \port\  for bucket $(b,\alpha)$-plane oriented increasing trees.

\section{Description of bucket increasing trees\label{secDESC}}

\subsection{Tree evolution processes\label{secDESCsub1}} 
In~\cite{MahSmy1995} a stochastic growth rule for bucket recursive trees has been given, generalizing the rule for ordinary random recursive trees (which are the special instance $b=1$): a bucket tree grows by progressive attraction of increasing integer labels: when inserting element $n+1$ into an existing bucket recursive tree containing $n$ elements (i.e., containing the labels $\{1, 2, \dots, n\}$) all $n$ existing elements in the tree compete to attract the element $n+1$, where all existing elements have equal chance to recruit the new element. If the element winning this competition is contained in a node with less than $b$ elements (an unsaturated bucket or node), element $n+1$ is added to this node, otherwise if the winning element is contained in a node with already $b$ elements (a saturated bucket or node), element $n+1$ is attached to this node as a new bucket containing only the element $n+1$.
Starting with a single bucket as root node containing only element $1$ leads after $n-1$ insertion steps,
where the labels $2, 3, \dots, n$ are successively inserted according to this growth rule, to a so called random bucket recursive tree with $n$ elements and maximal bucket size $b$. Of course, the above growth rule for inserting the element $n+1$ could also be formulated by saying that, for an existing bucket recursive tree $T$ with $n$ elements, the probability that a certain node $v \in T$ attracts the new element $n+1$ is proportional to the number of elements contained in $v$, let us say $k$ with $1 \le k \le b$, and is thus given by $\frac{k}{n}$. 
Summarizing this procedure we obtain the following definition. 

\begin{defi}[Bucket recursive trees]
\label{def0}
We start with a single bucket as root node containing only label $1$. 
Given a tree $T$ of size $n\ge 1$. Let $p(v)=\P(n+1 <_t v \mid c(v)\}$ denote the probability that node $v\in T$ attracts label $n+1$ conditioned on its capacity $c(v)$.
The family of random \emph{bucket recursive trees} is generated according to the probabilities 
\[
   p(v)  = \frac{c(v)}{n},
\] 
with capacity $1\le c(v)\le  b$, thus independent of the out-degree $\grad^{+}(v)\ge 0$ of node $v$.
\end{defi}
Note that here and throughout this work the capacities $c(v)=c_n(v)$ and  
the out-degree $\grad^{+}(v)=\grad^{+}_n(v)$ of a node $v$ in a tree $T$ are always dependent on the size $|T|=n$.

\smallskip

In the following we present two new stochastic growth rules. They generate families of random bucket increasing trees with bucket size $b\ge 1$. Concerning the created trees, they are by definition \emph{unordered}. For both rules there is an additional dependence on the out-degree $\grad^{+}(v)$ of the nodes.

\begin{defi}[\Dit]
\label{def1}
We start, case $n=1$, with a single bucket as root node containing only label $1$. 
Given a tree $T$ of size $n\ge 1$. Let $p(v)$ denote the probability that node $v\in T$ attracts label $n+1$ in a bucket increasing tree of size $n\in\N$. 

\smallskip

The family of random \emph{(b,d)-ary increasing trees}, with $d\in\N\setminus\{1\}$, is generated according to the probabilities
\[
   p(v)  =\displaystyle{ \frac{(d-1)c(v)+1-\grad^{+}(v)}{(d-1)n+1}}, 
\]
with $1\le c(v)\le  b$ and $\grad^{+}(v)\ge 0$.
\end{defi}

\smallskip

\begin{defi}[\Port]
\label{def2}
We start, case $n=1$, with a single bucket as root node containing only label $1$. 
Given a tree $T$ of size $n\ge 1$. The family of random \emph{(b,$\alpha$)-plane oriented recursive trees}, with $\alpha>0$, is generated according to the probabilities $p(v)$ that node $v\in T$ attracts label $n+1$:
\[
   p(v)  =  \displaystyle{ \frac{\grad^{+}(v)+(\alpha+1)c(v)-1}{(\alpha+1)n-1}},
\]
with $1\le c(v)\le  b$ and $\grad^{+}(v)\ge 0$.
\end{defi}

\begin{remark}
Panholzer and Prodinger~\cite{PanPro2007} characterized increasing tree families, which can be constructed by a simple stochastic growth rule. They obtained a unifying description based on two parameters. They obtained three different families of trees. The two growth processes for \Dit\ and \port\, together with the process for random bucket recursive trees stated before, generalize the processes of~\cite{PanPro2007}, namely, they correspond to the special case of bucket size $b=1$. 
\end{remark}

\begin{remark}[Linear bucket increasing trees]
Plane-oriented recursive trees (PORTs, i.e., $(1,1)$-PORTs) and their generalization also appeared in the literature under different names: Pittel~\cite{Pittel} calls such tree families linear increasing trees, also unifying recursive trees and $d$-ary increasing trees into a single family. Closely related tree families are so-called scale-free trees, which can also be generated by a preferential attachment rule.

\smallskip

In the spirit of Pittel, we can define \myt{linear bucket increasing trees} as follows. 
Given a tree $T$ of size $n\ge  1$, a label $n+1$
is attracted by a node $v\in T$ that is chosen with probability proportional to 
\[
\alpha\cdot (c(v)-1)+\beta\cdot\grad^{+}(v)+m,
\]
with $1\le c(v)\le  b$ and $\grad^{+}(v)\ge 0$. Here $\alpha,\beta,m\in\R$ denote real parameters such that the sum is non-negative.
For $b=1$ we have $c(v)=1$ and we reobtain Pittel's linear increasing trees. 
\end{remark}



\subsection{Combinatorial description of bucket increasing tree families}
Our basic objects are rooted ordered trees $T\in\mathcal{B}$. Here, the order of the
subtrees of a node is of relevance. The nodes of the trees are buckets
with an integer capacity $c$, with $1 \le c \le b$ for a given
maximal integer bucket size $b \ge 1$. We assume that all internal nodes (i.e., non-leaves) in the tree
must be saturated, while the leaves might be either saturated or
unsaturated. Here $\mathcal{B}$ denotes the family of all bucket ordered trees with maximal bucket size $b$.
A tree $T$ defined in this way is called a bucket ordered tree with maximal bucket size $b$. As already mentioned we define for bucket ordered trees the size $|T|$ of a tree $T$ via $|T| = \sum_{v} c(v)$, where $c(v)$ ranges over
all vertices of $T$. An increasing labelling $\ell(T)$ of a bucket ordered
tree $T$ is then a labelling of $T$, where the labels $\{1, 2,
\dots, |T|\}$ are distributed amongst the nodes of $T$, such that
the following conditions are satisfied: $(i)$ every node $v$
contains exactly $c(v)$ labels, $(ii)$ the labels within a node are
arranged in increasing order, $(iii)$ each sequence of labels along
any path starting at the root is increasing. A bucket ordered increasing tree $\tilde{T}$ is given by a pair $\tilde{T}=(T,\ell(T))$. 

Then a class $\mathcal{T}$ of a family of bucket increasing trees with maximal bucket size $b$ can be defined
in the following way. A sequence of non-negative numbers $(\varphi_{k})_{k \ge 0}$ with $\varphi_{0} > 0$
and a sequence of non-negative numbers $\psi_{1}, \psi_{2}, \dots, \psi_{b-1}$
is used to define the weight $w(T)$ of any bucket ordered tree $T$ by $w(T) := \prod_{v} w(v)$, where $v$ ranges over all vertices of $T$. The weight $w(v)$ of a node $v$ is given as follows, where $\grad^{+}(v)$ denotes the out-degree (i.e., the number of children) of node $v$:
\begin{equation*}
   w(v) =
   \begin{cases}
      \varphi_{\grad^{+}(v)}, & \quad \text{if} \enspace c(v)=b, \\
      \psi_{c(v)}, & \quad \text{if} \enspace c(v) < b.
   \end{cases}
\end{equation*}
Thus, for saturated nodes the weight depends on the out-degree $\grad^{+}(v)$ and is described by the sequence $\varphi_{k}$, whereas for unsaturated nodes the weight depends on the capacity $c(v)$ and is described by the sequence $\psi_{k}$.

Furthermore, $\mathcal{L}(T)$ denotes the set of different increasing labellings $\ell(T)$ of the tree $T$ with distinct integers $\{1, 2, \dots, |T|\}$, where $L(T) := \big|\mathcal{L}(T)\big|$ denotes its cardinality.
The family $\mathcal{T}$ consists of all trees $\tilde{T}=(T,\ell(T))$, with their weights $w(T)$ and the set of increasing labellings $\mathcal{L}(T)$ and we define $w(\tilde{T}):=w(T)$. Concerning bucket ordered increasing trees, note that the left-to-right order of the subtrees of the nodes is relevant. E.g., the trees \raisebox{-0.6em}{\includegraphics[height=2em]{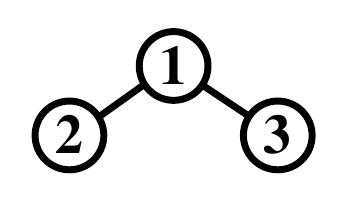}} and
\raisebox{-0.6em}{\includegraphics[height=2em]{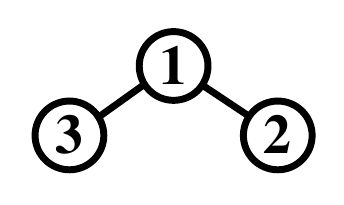}} are forming two different trees.

For a given degree-weight sequence $(\varphi_{k})_{k \ge 0}$ with a degree-weight generating function
$\varphi(t) := \sum_{k \ge 0} \varphi_{k} t^{k}$ and a bucket-weight sequence $\psi_{1}, \dots, \psi_{b-1}$,
we define now the total weights $T_n$ by 
\[
T_{n} :=  \sum_{T\in\mathcal{B}\colon |T|=n}w(T)\cdot L(T)=\sum_{\tilde{T}=(T,\ell(T))\in\mathcal{T}\colon |T|=n} w(\tilde{T}).
\]

\medskip

It is advantageous for such enumeration problems to describe a family of increasing trees $\mathcal{T}$ by the following
formal recursive equation:
\begin{align}
   \mathcal{T} & = \psi_{1} \cdot \raisebox{-0.5ex}{\includegraphics[height=2.3ex]{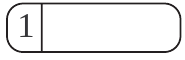}} \; \dot{\cup} \;
   \psi_{2} \cdot \raisebox{-0.5ex}{\includegraphics[height=2.3ex]{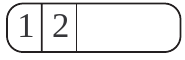}} \; \dot{\cup} \; \cdots \; \dot{\cup} \;
   \psi_{b-1} \cdot \raisebox{-0.5ex}{\includegraphics[height=2.3ex]{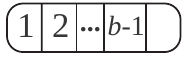}} \; \dot{\cup} \notag \\
   & \quad \varphi_{0} \cdot \raisebox{-0.5ex}{\includegraphics[height=2.3ex]{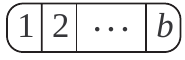}} \; \dot{\cup} \;
   \varphi_{1} \cdot \raisebox{-0.5ex}{\includegraphics[height=2.3ex]{bucket4.pdf}} \times \mathcal{T} \; \dot{\cup} \;
   \varphi_{2} \cdot \raisebox{-0.5ex}{\includegraphics[height=2.3ex]{bucket4.pdf}} \times \mathcal{T} \ast \mathcal{T} \; \dot{\cup} \;
   \varphi_{3} \cdot \raisebox{-0.5ex}{\includegraphics[height=2.3ex]{bucket4.pdf}} \times \mathcal{T} \ast \mathcal{T} \ast \mathcal{T} \;
   \dot{\cup} \; \cdots \label{eqna2} \\
   & = \psi_{1} \cdot \raisebox{-0.5ex}{\includegraphics[height=2.3ex]{bucket1.pdf}} \; \dot{\cup} \;
   \psi_{2} \cdot \raisebox{-0.5ex}{\includegraphics[height=2.3ex]{bucket2.pdf}} \; \dot{\cup} \; \cdots \; \dot{\cup} \;
   \psi_{b-1} \cdot \raisebox{-0.5ex}{\includegraphics[height=2.3ex]{bucket3.pdf}} \; \dot{\cup} \;
   \raisebox{-0.5ex}{\includegraphics[height=2.3ex]{bucket4.pdf}} \times \varphi\big(\mathcal{T}\big), \notag
\end{align}
where $\raisebox{-0.5ex}{\includegraphics[height=2.3ex]{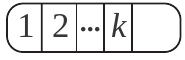}}$ denotes a bucket of capacity $k$ labelled by $1, 2, \dots, k$,
$\times$ the cartesian product, $\ast$ the partition product for labelled
objects, and $\varphi(\mathcal{T})$ the substituted structure. On the other hand, we can use standard notation~\cite{FlaSed}. Let $\mathcal{Z}$ denote the \emph{atomic class} (i.e., a single (uni)labelled node), $\mathcal{A}^{\Box} \ast \mathcal{B}$ the \emph{boxed product} (i.e., the smallest label is constrained to lie in the $\mathcal{A}$ component) of the combinatorial classes $\mathcal{A}$ and $\mathcal{B}$. Then, 
\begin{align*}
 \mathcal{T} & = \psi_1\cdot \mathcal{Z}^{\Box} + \psi_2\cdot (\mathcal{Z}^{\Box})^2+ \dots
+ \psi_{b-1}\cdot (\mathcal{Z}^{\Box})^{b-1} + (\mathcal{Z}^{\Box})^{b}\ast \varphi(\mathcal{T})\\
 & = \sum_{k=1}^{b-1}\psi_{k}\cdot (\mathcal{Z}^{\Box})^{k} + (\mathcal{Z}^{\Box})^{b}\ast \varphi(\mathcal{T}).
\end{align*}
Here the meaning of $(\mathcal{Z}^{\Box})^{k} \ast \mathcal{B}$ is $\mathcal{Z}^{\Box} \ast \left(\mathcal{Z}^{\Box} \ast \left(\cdots \ast \left(\mathcal{Z}^{\Box} \ast \mathcal{B}\right)\right)\right)$, with $k$ occurrences of $\mathcal{Z}^{\Box}$.

\medskip

Using above formal description, one can show that the exponential generating function
$T(z) := \sum_{n \ge 1} T_{n} \frac{z^{n}}{n!}$ of the total weights $T_{n}$
is characterized by the following result of~\cite{BucketPanKu2009}.
\begin{prop}
The exponential generating function $T(z)$ of bucket increasing trees with degree-weight generating function $\varphi(t)$ satisfies 
an ordinary differential equation of order $b$:
\begin{align}
   \label{eqna1}
   \frac{d^{b}}{d z^{b}} T(z) & = \varphi\big(T(z)\big), 
\end{align}
with initial conditions
\[
   T(0)=0, \qquad T^{(k)}(0) = \psi_{k}, \quad \text{for} \enspace 1 \le k \le b-1.
\]
\end{prop}

\begin{example}[Bucket size one -- ordinary increasing trees]
In the case of bucket size $b=1$ we obtain ordinary increasing trees. 
We have a simple description using the \emph{boxed product}:
\begin{equation}
  \mathcal{T} = \mathcal{Z}^{\Box}\ast \varphi\big(\mathcal{T}\big).
\end{equation}

The exponential generating function $T(z)$ of bucket increasing trees with degree-weight generating function $\varphi(t)$ is implicitly defined by
\[
\int_0^{T(z)}\frac{dt}{\varphi(t)} = z.
\]
Prominent varieties include recursive trees, binary increasing and plane-oriented recursive trees; for an overview see, e.g.,~\cite{Drmota,FlaSed}. 
\end{example}

\begin{example}[Bucket size two -- bilabelled trees]\label{example:bilabelled}
Trees with bucket size $b=2$, degree-weight generating function $\varphi(t)$ and weight $\psi=\psi_1\ge 0$ correspond to so-called bilabelled trees. The family $\widehat{\mathcal{T}}$ can be described by the following symbolic equation:
\begin{equation}
\label{HookBijII-eqn1}
  \mathcal{T} = \mathcal{Z}^{\Box} \ast \left(\mathcal{Z}^{\Box} \ast \varphi\big(\mathcal{T}\big)\right),
\end{equation}
The differential equation $T''(z)=\varphi(T(z))$ can be readily translated to a first-order equation. 
Namely, multiplication with $T'(z)$ and integration leads to the first-order differential equation
\begin{equation}
\label{HookBijII-DGL2}
    T'(z) = \sqrt{\psi^2+2\cdot\Phi(T(z))}, \quad T(0)=0,
\end{equation}
with $\Phi(x)=\int_{0}^{x} \varphi(t)dt$. Hence, $T=T(z)$ is implicitly given via
\[
\int_{0}^{T}\frac{dx}{\sqrt{\psi^2+2\cdot\Phi(x)}}=z.
\]
The special choice $\psi=0$ leads to so-called strict bilabelled increasing tree families. 
Here, due to the choice $\psi=0$, all nodes have to be saturated. Such tree families naturally give rise to hook-length formulas; see~\cite{KubPan2016} for many examples.
\end{example}

\begin{example}[Strict $b$-labelled tree families]
Families of strictly $b$-labelled trees have been recently studied due to their connection to hook-length formulas~\cite{KubPan2012,KubPan2016}. Here $\psi_1=\dots=\psi_{b-1}=0$, thus only saturated nodes are allowed.  
\end{example}

\section{Combinatorial models of the tree evolution processes}
\subsection{Random tree models}
Given a degree-weight sequence $(\varphi_{k})_{k \ge 0}$ and a bucket-weight sequence $\psi_{1}, \dots, \psi_{b-1}$. 
For each class $\mathcal{T}$ of bucket increasing trees associated to $(\varphi_{k})_{k \ge 0}$ and $\psi_{1}, \dots, \psi_{b-1}$
we define in a natural way probability models for random (ordered) bucket increasing trees $\mathcal{T}_n$ of size $n$. 
We assume that each increasingly labelled bucket ordered increasing tree $\tilde{T} = \big(T,\ell(T)\big)\in\mathcal{T}_n$ of size $n$ is chosen with a probability proportional to its weight $w(\tilde{T})$. 

\begin{defi}[Random bucket ordered increasing trees]
A probability measure on ordered bucket increasing trees $\tilde{T}\in\mathcal{T}_n$ of size $n$ 
is defined by 
\begin{equation*}
   \P\ptree(\big\{\tilde{T}\big\}) = \frac{w(\tilde{T})}{T_{n}}
	=\frac{\displaystyle{\bigg(\prod_{v\in \tilde{T}\colon c(v)<b} \psi_{c(v)} \bigg) \cdot \bigg(\prod_{v\in \tilde{T}\colon c(v)=b}\varphi_{\grad^{+}(v)}\bigg)}}{T_{n}}.
\end{equation*}
We speak then about \myt{random ordered bucket increasing trees} of size $n$ of the family $\mathcal{T}$ under the \myt{random tree model}.
\end{defi}

\medskip

From a combinatorial point of view it is often convenient to work with ordered trees. However, the trees generated by the tree evolution processes are by definition unordered. Thus, we turn our attention to unordered trees. Our basic objects are unordered bucket trees $T\unord$ together with an increasing labelling $\ell(T\unord)$.  In order to obtain a measure on unordered bucket increasing trees $\mathcal{T}\unord$ 
we proceed in a standard way embedding unordered trees into ordered trees. Each unordered bucket increasing tree $\tilde{T}\unord =(T\unord,\ell(T\unord))$ corresponds to 
\[
\prod_{v\in T\unord}\grad^{+}(v)! 
\]
different ordered bucket increasing trees. We obtain a canonical ordered representative $\tilde{T}\in\mathcal{T}$ of the unordered tree $\tilde{T}\unord=(T\unord,\ell(T\unord))$ by ordering the subtrees of each node $v\in\tilde{T}\unord$ in an increasing left-to-right way according to the smallest labels contained in the respective subtrees taking into account the labelling $\ell(T\unord)$. Let $f\colon \mathcal{T}\unord\to \mathcal{T}$ denote this injective ordering map. 

\smallskip

\begin{defi}[Random unordered bucket increasing trees]
A probability measure $\P\ptree$ on unordered bucket increasing trees $\tilde{T}\unord\in\mathcal{T}\unord_{n}$ of size $n$ 
is defined by 
\begin{equation*}
   \P\ptree\big\{\tilde{T}\unord\big\} = \frac{(w\circ f)(\tilde{T}\unord) \cdot \prod_{v\in T\unord}\grad^{+}(v)! }{T_{n}}.
\end{equation*}
We speak then about \myt{random unordered bucket increasing trees} of size $n$ of the family $\mathcal{T}\unord$ under the \myt{random tree model}.
\end{defi}

\medskip

Possibly different models of randomness are introduced when generating an unordered tree of size $n$ according to one of the three tree evolution processes described before. 
Given an integer $j\ge 2$ and any tree $\tilde{T}\unord=(T\unord,\ell(T\unord))$, we denote by $\text{attr}(j)$ the node in $\tilde{T}\unord$ that attracted label $j$, and by $\tilde{T}_{< j}\unord$ the tree obtained from $\tilde{T}\unord$ when restricting to labels less than $j$.
Then we define $\pi^{(j)}$ as the map $\pi^{(j)}\colon \tilde{T}\unord \to [0,1]$ that gives the probability that label $j$ will be attracted by node $\text{attr}(j)$ in $\tilde{T}_{< j}\unord$:
\[
\pi^{(j)}(\tilde{T}\unord)=
\begin{cases}
\P\big\{j<_t \text{attr}(j) \: \big| \: \tilde{T}_{< j} \unord\big\},\qquad  |\tilde{T}\unord|\ge j,\\
0, \qquad |\tilde{T}\unord|<j,
\end{cases}
\]
with $\P\big\{j<_t \text{attr}(j) \: \big| \: \tilde{T}_{< j}\unord\big\}$ being determined by the tree evolution processes.

\smallskip 

The point-wise product $\P^{[e]}=\prod_{j=2}^{n}\pi^{(j)}$ is then a probability measure on unordered bucket increasing trees $\mathcal{T}\unord_{n}$ of size $n$.
In the following we denote with a superscript the source of randomness on $\mathcal{T}_n$: $\P\ptree$ for the random tree model and $\P\pgrow$ for the tree evolution process. It holds
\[
\P\pgrow\{\tilde{T}\unord\}=\prod_{j=2}^{n}\pi^{(j)}(\tilde{T}\unord).
\] 

Let $\grad^{+}_{< j}(v)$ be the out-degree of node $v$ when restricting to labels less than $j$. 
Then, a probability measure on ordered trees is obtained via
\[
\P\pgrow\{\tilde{T}\}=\prod_{j=2}^{n}\frac{\pi^{(j)}(\tilde{T}\unord)}{\grad^{+}_{< j}(\text{attr}(j))}.
\]

\subsection{Combinatorial models}
It was already proven in~\cite{BucketPanKu2009} that bucket recursive trees 
generated according to the growth process as stated in Subsection~\ref{secDESCsub1} can be considered as a certain bucket increasing tree family. In the following we extend this result, where, for the sake of completeness, we also collect the findings of~\cite{BucketPanKu2009} for bucket recursive trees.
\begin{theorem}[Combinatorial models for families generated by a stochastic growth rule]
\label{BINCprop1}
The tree evolution processes that generate families of random unordered bucket recursive trees, \Dit{} and \Port, with bucket size $b\ge 1$, can be realized combinatorially by suitably chosen sequences of degree-weights $(\varphi_{k})_{k\ge 0}$ and bucket weights $\psi_{1}, \dots, \psi_{b-1}$. Given an arbitrary ordered bucket increasing tree $T\in\mathcal{T}$ of size $|T| = n$, then it holds that under the random tree model the probability that a new
element $n + 1$ is attracted by a node $v \in T$ with capacity $c(v) = k$ is given by $p(v)=\P\{n+1<_t v\mid c(v)=k\}$ as defined in Definitions~\ref{def0}-\ref{def2} for the corresponding tree evolution process.

\smallskip 

Consequently, both models of randomness for (un)ordered trees coincide: $\P\pgrow=\P\ptree$.

\smallskip
\begin{enumerate}
	\item Bucket recursive trees: a combinatorial model can be obtained from the sequences
\begin{equation*}
   \varphi_{k} = \frac{(b-1)! b^{k}}{k!}, \quad \text{for} \enspace k \ge 0, \qquad \psi_{k} = (k-1)!, \quad \text{for} \enspace 1 \le k \le b-1,
\end{equation*}
such that $\varphi(t)=\sum_{k\ge 0}\varphi_k t^k=(b-1)! \cdot \exp(b t)$.
For this model of bucket recursive trees the exponential generating function $T(z)$ and the total weights $T_n$ are given by
\begin{equation*}
T(z)=\log\big(\frac{1}{1-z}\big),\quad T_n=(n-1)!.
\end{equation*}

\smallskip

\item \Dit: a combinatorial model can be obtained from the sequences
\begin{equation*}
\begin{split}
   \varphi_{k} &= (b-1)!(d-1)^{b-1}\binom{b-1+\frac{1}{d-1}}{b-1}\binom{b(d-1)+1}{k}, \quad \text{for} \enspace k \ge 0,\\
   \psi_{k} &= (k-1)!(d-1)^{k-1}\binom{k-1+\frac{1}{d-1}}{k-1}, \quad \text{for} \enspace 1 \le k \le b-1,
\end{split}
\end{equation*}
such that $\varphi(t)= (b-1)!(d-1)^{n-1}\binom{b-1+\frac{1}{d-1}}{b-1}(1+t)^{b(d-1)+1}$.
For this model of \Dit{} the exponential generating function $T(z)$ and the total weights $T_n=[z^n]T(z)$ are given by
\begin{equation*}
 T(z)=\frac{1}{(1-(d-1)z)^{\frac1{d-1}}}-1,\quad T_n=(n-1)!(d-1)^{n-1}\binom{n-1+\frac{1}{d-1}}{n-1}.
\end{equation*}

\smallskip

\item \Port: a combinatorial model can be obtained from the sequences
\begin{equation*}
\begin{split}
   \varphi_{k} &=(b-1)!(\alpha+1)^{b-1}\binom{b-1-\frac{1}{\alpha+1}}{b-1}\binom{(\alpha+1)b-2+k}{k}, \quad \text{for} \enspace k \ge 0,\\
   \psi_{k} &= (k-1)!(\alpha+1)^{k-1}\binom{k-1-\frac{1}{\alpha+1}}{k-1}, \quad \text{for} \enspace 1 \le k \le b-1,
\end{split}
\end{equation*}
such that $\varphi(t)=\frac{(b-1)!(\alpha+1)^{b-1}\binom{b-1-\frac{1}{\alpha+1}}{b-1}}{(1-t)^{(\alpha+1)b-1}}$.
For this model of \Port{} the exponential generating function $T(z)$ and the total weights $T_n=[z^n]T(z)$ are given by
\begin{equation*}
T(z)=1-(1-(\alpha+1)z)^{\frac1{\alpha+1}},\quad T_n=(n-1)!(\alpha+1)^{n-1}\binom{n-1-\frac{1}{\alpha+1}}{n-1}.
\end{equation*}
\end{enumerate}
\end{theorem}


\begin{proof}
To prove that these choices of sequences $(\varphi_k)_{k\in\N}$ and $(\psi_k)_{1\le k\le b-1}$ are actually models for bucket increasing trees generated according to the stochastic growth rules defined
in Subsection~\ref{secDESCsub1}, we have to show that the combinatorial families $\mathcal{T}$ of
bucket increasing trees have the same stochastic growth rules as the counterparts created probabilistically. 
Given an arbitrary bucket increasing tree $T \in \mathcal{T}$ of size $|T|=n$, 
then the probability that a new element $n+1$ is attracted by a node $v \in T$ with capacity $c(v)=k$
has to coincide with the corresponding probability stated in Definitions~\ref{def0}-\ref{def2}.

We use now the notation $T \to T'$ to denote that $T'$ is obtained from $T$ with $|T|=n$ by incorporating element $n+1$, i.e., either by attaching element $n+1$ to a saturated node $v \in T$ at one of the $\grad^{+}(v)+1$ possible positions
(recall that bucket increasing trees are per definition ordered trees and thus the order of the subtrees is of relevance)
by creating a new bucket of capacity $1$ containing element $n+1$ or by adding element $n+1$ to an unsaturated node $v \in T$ by increasing
the capacity of $v$ by $1$. If we want to express that node $v \in T$ has attracted the element $n+1$ leading from $T$ to $T'$ we use
the notation $T \xrightarrow{v} T'$. If there exists a stochastic growth rule for a bucket increasing tree family $\mathcal{T}$,
then it must hold that for a given tree $T \in \mathcal{T}$ of size $|T|=n$ and a given node $v \in T$ the probability $p_{T}(v)$, which gives
the probability that element $n+1$ is attracted by node $v \in T$, is given as follows:
\begin{equation}
   p_{T}(v) = \frac{\sum_{T' \in \mathcal{T} : T \xrightarrow{v} T'} w(T')}{\sum_{\tilde{T} \in \mathcal{T} : T \to \tilde{T}} w(\tilde{T})}
   = \frac{\sum_{T' \in \mathcal{T} : T \xrightarrow{v} T'} \frac{w(T')}{w(T)}}{\sum_{\tilde{T} \in \mathcal{T} : T \to \tilde{T}}
   \frac{w(\tilde{T})}{w(T)}}.
\end{equation}
The remaining task is to simplify the expression above into the form stated in Definitions~\ref{def0}-\ref{def2}. For a certain tree $\tilde{T}$ with $T \xrightarrow{u} \tilde{T}$ and $u \in T$ the quotient of the weight of the trees $\tilde{T}$ and $T$
is due to the definition of bucket increasing trees given as follows, where we define for simplicity $\psi_{b} := \varphi_{0}$:
\begin{equation*}
   \frac{w(\tilde{T})}{w(T)} =
   \begin{cases}
      \psi_{1} \frac{\varphi_{k+1}}{\varphi_{k}}, & \quad \text{for} \enspace c(u)=b \quad \text{and} \quad \grad^{+}(u)=k, \\
      \frac{\psi_{k+1}}{\psi_{k}}, & \quad \text{for} \enspace c(u)=k < b.
   \end{cases}
\end{equation*}
For a given tree $T \in \mathcal{T}$ we define by $m_{k} := |\{u \in T : c(u)=k<b\}|$ the number of unsaturated nodes of $T$ with capacity
$k<b$ and by $n_{k} := |\{u \in T : c(u)=b \enspace \text{and} \enspace \grad^{+}(u)=k\}|$ the number of saturated nodes of $T$ with out-degree $k \ge 0$. It holds then
\begin{equation}
\label{BINCeqn001}
   n = \sum_{u \in T} c(u) = \sum_{k=1}^{b-1} k m_{k} + b \sum_{k \ge 0} n_{k}
\end{equation}
and (where we use that there are $k+1$ possibilities of attaching a new node to a saturated node $u \in T$ with out-degree
$\grad^{+}(u)=k$):
\begin{equation*}
   \sum_{\tilde{T} \in \mathcal{T} : T \to \tilde{T}} \frac{w(\tilde{T})}{w(T)}
   = \sum_{k=1}^{b-1} m_{k} \frac{\psi_{k+1}}{\psi_{k}} + \sum_{k \ge 0} n_{k} (k+1) \psi_{1} \frac{\varphi_{k+1}}{\varphi_{k}}.
\end{equation*}
Moreover, we also have the relation
\begin{equation}
\label{BINCeqn002}
 1 = \sum_{k=1}^{b-1} m_{k} - \sum_{k \ge 0} (k-1)n_{k},
\end{equation}
which follows as the difference between the node-sum and edge-sum equation for the tree $T$:
\begin{equation*}
\text{$\#$ nodes} = \sum_{k=1}^{b-1} m_{k} + \sum_{k \ge 0} n_{k}, \qquad \text{$\#$ edges} = \text{$\#$ nodes} - 1 = \sum_{k \ge 0} k n_{k}.
\end{equation*}

First we turn our attention to the family of \Dit{} and the weights as given in Theorem~\ref{BINCprop1}.
We have 
\begin{equation*}
   \sum_{T' \in \mathcal{T} : T \xrightarrow{v} T'} \frac{w(T')}{w(T)} =
   \begin{cases}
      (k+1) \psi_{1} \frac{\varphi_{k+1}}{\varphi_{k}} = b(d-1)+1-k, & \quad \text{for} \enspace c(v)=b \quad \text{and} \quad \grad^+(v)=k, \\
      \frac{\psi_{k+1}}{\psi_{k}} =  k(d-1)+1, & \quad \text{for} \enspace c(v)=k < b,
   \end{cases}
\end{equation*}
and consequently
\begin{equation*}
\begin{split}
   \sum_{\tilde{T} \in \mathcal{T} : T \to \tilde{T}} \frac{w(\tilde{T})}{w(T)} &=
   \sum_{k=1}^{b-1} (k(d-1)+1) m_{k} + \sum_{k \ge 0} n_{k} (b(d-1)+1-k)\\
   &= (d-1)\big(\sum_{k=1}^{b-1} k m_{k} + b \sum_{k \ge 0} n_{k}\big) + \sum_{k=1}^{b-1} m_{k} - \sum_{k \ge 0} (k-1)n_{k} =(d-1)n+1,
\end{split}
\end{equation*}
due to equations~\eqref{BINCeqn001} and \eqref{BINCeqn002}.
Thus, with this choice of weight sequences $(\varphi_k)_k$ and $(\psi_k)_k$, 
the probability $p_{T}(v)$ that in a bucket increasing tree $T$ of size $|T|=n$ the node $v$ with capacity $c(v)=k$ attracts element $n+1$
coincides with the corresponding probability of the stochastic growth rule for \Dit{} given in Definition~\ref{def1}.

We obtain then from equation~\eqref{eqna1} that the exponential generating function $T(z):=\sum_{n\ge 1}T_n\frac{z^{n}}{n!}$
of the total-weight $T_{n}$ of bucket increasing trees of size $n$ satisfies the differential equation
\begin{equation}
   \frac{d^b}{dz^b}T(z)=(b-1)!(d-1)^{b-1}\binom{b-1+\frac{1}{d-1}}{b-1} (1+T(z))^{b(d-1)+1)},
\end{equation}
with initial conditions $T(0)=0$ and $\left.\frac{d^k}{dz^k}T(z)\right|_{z=0}=(k-1)!(d-1)^{k-1}\binom{k-1+\frac{1}{d-1}}{k-1}$, for $1\le k \le b-1$.
The solution of this equation is given by
\begin{equation}
   T(z)=\frac{1}{(1-(d-1)z)^{\frac1{d-1}}}-1=\sum_{n\ge 1}(n-1)!(d-1)^{n-1}\binom{n-1+\frac{1}{d-1}}{n-1} \frac{z^{n}}{n!},
\end{equation}
as can be checked easily
Hence the total weight of all size $n$ \Dit{} is given by $T_n=(n-1)!(d-1)^{n-1}\binom{n-1+\frac{1}{d-1}}{n-1}$ as stated in Theorem~\ref{BINCprop1}.

\smallskip

For \Port{} and weights as given in Theorem~\ref{BINCprop1} we obtain
\begin{equation*}
   \sum_{T' \in \mathcal{T} : T \xrightarrow{v} T'} \frac{w(T')}{w(T)} =
   \begin{cases}
      (k+1) \psi_{1} \frac{\varphi_{k+1}}{\varphi_{k}} = (\alpha+1)b-1+k, & \quad \text{for} \enspace c(v)=b \quad \text{and} \quad \grad^+(v)=k, \\
      \frac{\psi_{k+1}}{\psi_{k}} =  k(\alpha+1)-1, & \quad \text{for} \enspace c(v)=k < b,
   \end{cases}
\end{equation*}
and consequently
\begin{equation*}
\begin{split}
   \sum_{\tilde{T} \in \mathcal{T} : T \to \tilde{T}} \frac{w(\tilde{T})}{w(T)} &=
   \sum_{k=1}^{b-1} (k(\alpha+1)-1) m_{k} + \sum_{k \ge 0} n_{k} ((\alpha+1)b-1+k)\\
   &= (\alpha+1)\big(\sum_{k=1}^{b-1} k m_{k} + b \sum_{k \ge 0} n_{k}\big) - \big(\sum_{k=1}^{b-1} m_{k} - \sum_{k \ge 0} (k-1)n_{k}\big) =(\alpha+1)n-1,
\end{split}
\end{equation*}
due to equations~\eqref{BINCeqn001} and \eqref{BINCeqn002}.
Again, with this choice of weight sequences $(\varphi_k)_k$ and $(\psi_k)_k$, it follows that
the probability $p_{T}(v)$ that in a bucket increasing tree $T$ of size $|T|=n$ the node $v$ with capacity $c(v)=k$ attracts element $n+1$
coincides with the corresponding probability in the stochastic growth rule for \Port{} given in Definition~\ref{def2}.

We obtain then from equation~\eqref{eqna1} that the exponential generating function $T(z):=\sum_{n\ge 1}T_n\frac{z^{n}}{n!}$
of the total-weight $T_{n}$ of \Port{} of size $n$ satisfies the differential equation
\begin{equation}
   \frac{d^b}{dz^b}T(z)=(b-1)!(\alpha+1)^{b-1}\binom{b-1-\frac{1}{\alpha+1}}{b-1} \frac{1}{(1-T(z))^{b(\alpha+1)-1}},
\end{equation}
with initial conditions $T(0)=0$ and $\left.\frac{d^k}{dz^k}T(z)\right|_{z=0}=(k-1)!(\alpha+1)^{k-1}\binom{k-1-\frac{1}{\alpha+1}}{k-1}$, for $1\le k \le b-1$.
Again it can be checked easily that the solution of this equation is given by
\begin{equation}
   T(z)=(1-(\alpha+1)z)^{\frac1{\alpha+1}}-1=\sum_{n\ge 1}(n-1)!(\alpha+1)^{n-1}\binom{n-1-\frac{1}{\alpha+1}}{n-1} \frac{z^{n}}{n!}.
\end{equation}
Hence the total weight of all size $n$ \Port is given by $T_n=(n-1)!(\alpha+1)^{n-1}\binom{n-1-\frac{1}{\alpha+1}}{n-1}$,
which finishes the proof of Theorem~\ref{BINCprop1}
\end{proof}

In Tables~\ref{tab:BICgrowth}-\ref{tab:BICcomb} we summarize the combinatorial properties as well as the growth processes of the bucket increasing tree families considered in this work.
\begin{table}[!htb]
   \begin{center}
   \begin{tabular}{|c|c|c|c|}
   \hline
   \rule[-1ex]{0ex}{3.5ex}
	 \emph{Tree family} & Growth process : $p(v)$\\
   \hline
   \rule[-1.5ex]{0ex}{4.5ex}
	Bucket recursive trees & $\frac{c(v)}{n}$ \\
   \hline
   \rule[-2.7ex]{0ex}{6.5ex}
	 \emph{($b,d$)-ary increasing trees} & $\frac{(d-1)c(v)+1-\grad^{+}(v)}{(d-1)n+1}$  \\
   \hline
   \rule[-2.7ex]{0ex}{6.5ex}
	 \emph{($b,\alpha$)-PORT} & $\frac{\grad^{+}(v)+(\alpha+1)c(v)-1}{(\alpha+1)n-1}$ \\
   \hline
   \end{tabular}
   \caption{Summary: bucket increasing tree families and their growth processes.\label{tab:BICgrowth}}
   \end{center}
\end{table}
\begin{table}[!htb]
   \begin{center}
   \begin{tabular}{|c|c|c|}
   \hline
   \rule[-1ex]{0ex}{3.5ex}
	 \emph{Tree family} & \emph{Degree-weight GF $\varphi(t)$}  & \emph{initial weights} $\psi_k$, $1\le k\le b-1$  \\
   \hline
   \rule[-1.5ex]{0ex}{4.5ex}
	Bucket recursive trees  & $(b-1)!\exp(b\cdot t)$ 
   & $(k-1)!$  \\
   \hline
   \rule[-2.7ex]{0ex}{6.5ex}
	 \emph{($b,d$)-ary increasing trees} & $(b-1)!(d-1)^{n-1}\binom{b-1+\frac{1}{d-1}}{b-1}(1+t)^{b(d-1)+1}$ & 
	$(k-1)!(d-1)^{k-1}\binom{k-1+\frac{1}{d-1}}{k-1}$  \\
   \hline
   \rule[-2.7ex]{0ex}{6.5ex}
	 \emph{($b,\alpha$)-PORT} & $\frac{(b-1)!(\alpha+1)^{b-1}\binom{b-1-\frac{1}{\alpha+1}}{b-1}}{(1-t)^{(\alpha+1)b-1}}$ &
	$(k-1)!(\alpha+1)^{k-1}\binom{k-1-\frac{1}{\alpha+1}}{k-1}$  \\
   \hline
   \end{tabular}
   \caption{Summary: bucket increasing tree families and their combinatorial properties.\label{tab:BICcomb}}
   \end{center}
\end{table}

\section{Clustering process and weight sequences}
Up to this point, no indication has been given how to find the weight sequences in Theorem~\ref{BINCprop1}. 
Similarly, we did not give any motivation behind the definition of the growth processes. Here, as Algorithm~\ref{cluster} we will present
a clustering map  for ordinary increasing trees, case $b=1$. It leads directly to the stated weight sequences and 
serves as a motivation behind the definition of the growth processes and weight sequences.

\smallskip

Let $\mathcal{T}^{[b]}$ denote the family of ordered bucket increasing trees with maximal bucket size $b$, such that $\mathcal{T}^{[1]}$ denotes the family of ordinary ordered increasing trees. Algorithm~\ref{cluster}, which defines a map $\mathcal{C}\colon \mathcal{T}^{[1]}\to\mathcal{T}^{[b]}$, is given as follows; furthermore it is illustrated in Figure~\ref{ClusterFig1}.

\begin{algorithm}
\caption{ClusteringIncreasingTrees$(T,b)$}
\label{cluster}
\begin{algorithmic}[1]
\Input{Ordered Increasing tree $T\in\mathcal{T}^{[1]}$, integer $b\ge 2$}
\Result{Ordered bucket increasing tree $\tilde{T}=\mathcal{C}(T)\in\mathcal{T}^{[b]}$}
\State $V\gets V(T)$, $\tilde{T}\gets \emptyset$
\Do
	\State Choose $v_{\min}\in V$, the node with minimal label
	
	\If{$|\text{subtree}(v_{\min})|\ge b$} 
		\quad $s \gets b$
	\Else 
		\quad s$\gets |\text{subtree}(v_{\min})|$
\EndIf
    \State Merge $s$ smallest labeled nodes $v_{\min}=v_1,v_2,\dots,v_s\in \text{subtree}(v_{\min})$ into new bucket $v$.
		\State $V\gets V\setminus\{v_1,\dots,v_s\}$ 
		\State Redirect all the edges starting at any of these $s$ nodes to the new bucket.
		\State$\tilde{T}\gets \tilde{T}\cup\{v\}$ 
	\doWhile{$V\neq \emptyset$} 
\end{algorithmic}
\end{algorithm}

%
\begin{figure}[!htb]
\begin{center}
\includegraphics[height=4.5cm]{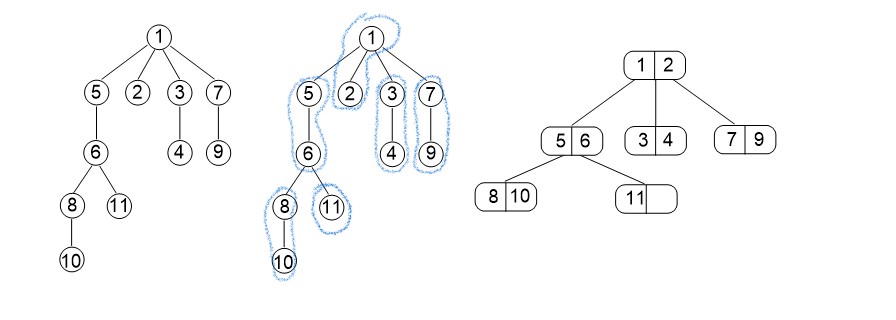}
\end{center}
\caption{A plane-oriented recursive tree $T$ of size eleven, the clustered tree with $b=2$ and the corresponding bilabelled increasing ordered tree $\tilde{T}=\mathcal{C}(T)$.\label{ClusterFig1} }
\end{figure}

\begin{prop}
The map $\mathcal{C}\colon \mathcal{T}^{[1]}\to\mathcal{T}^{[b]}$ defined in Algorithm~\ref{cluster} is not injective, but surjective.
\end{prop}

\begin{remark}
Note that for arbitrary increasing tree families the map $\mathcal{C}$ does neither take into account the weight sequence $(\phi_k)_{k\ge 0}$ of $\mathcal{T}^{[1]}$, nor it determines $(\varphi_k)_{k\ge 0}$ and $(\psi_k)_{1\le k\le b-1}$ of $\mathcal{T}^{[b]}$.
We will see later, see Examples~\ref{BIJex1} and ~\ref{BIJex2}, that the map can be modified to obtain bijections between ordinary increasing trees and specific families of bucket increasing trees.
\end{remark}

\begin{proof}
Apparently, all trees $T\in\mathcal{T}^{[1]}$ of size $k$, $1\le k\le b$, are mapped to a tree $\mathcal{C}(T)$ consisting only of a single bucket $v$ of capacity $c(v)=k$, thus the map is not injective. 
However, the map $\mathcal{C}$ is surjective: every size $n$ bucket increasing tree can be created by clustering an ordinary increasing tree. Given a bucket increasing tree $\tilde{T}\in\mathcal{T}^{[b]}$, we can replace all buckets 
by corresponding increasing chains, holding the corresponding labels. This gives an ordinary ordered increasing tree $T\in\mathcal{T}^{[1]}$ with $\mathcal{C}(T)=\tilde{T}$.
\end{proof}

\medskip

Given a family of ordinary increasing trees $\mathcal{T}^{[1]}$ with weights $(\phi_k)_{k\ge 0}$.
In order to obtain the weight sequences $(\varphi_k)_{k\ge 0}$ and $(\psi_k)_{1\le k\le b-1}$ of the bucketed families, 
we choose them in a weight preserving way.

\begin{defi}[Weight preserving bucket trees]
Given a family of ordinary increasing trees $\mathcal{T}^{[1]}$ with weights $(\phi_k)_{k\ge 0}$. We call a family $\mathcal{T}^{[b]}$  of bucket increasing trees with weights $(\varphi_k)_{k\ge 0}$ and $(\psi_k)_{1\le k\le b-1}$ \emph{weight preserving}, if it holds,
for all $\tilde{T}\in\mathcal{T}_n$,
\begin{equation}
\label{BINCCondition1}
w(\tilde{T})=\sum_{T\in \mathcal{T}^{[1]}\colon \mathcal{C}(T)=\tilde{T}}w(T).
\end{equation}
\end{defi}

\smallskip

Given $(\phi_k)_{k\ge 0}$, the total weights $T_n$ of the bucket increasing trees, with weight preserving $(\varphi_k)_{k\ge 0}$ and $(\psi_k)_{1\le k\le b-1}$, equals the weight $T_n$ of their ordinary increasing trees counterpart:
\[
\sum_{\tilde{T}\in \mathcal{T}^{[b]}, |\tilde{T}|=n}w(\tilde{T})
= \sum_{\tilde{T}\in \mathcal{T}^{[b]}, |\tilde{T}|=n}\sum_{T\in \mathcal{T}^{[1]}\colon \mathcal{C}(T)=\tilde{T}}w(T)
= \sum_{T\in \mathcal{T}^{[1]}, |T|=n} w(T) = T_n.
\]

Moreover, we directly obtain from~\eqref{BINCCondition1}
\[
\psi_k=T_k, \quad 1\le k\le b-1,\qquad \varphi_0=T_b,
\]
where $T_k$ denotes the total weight of size $k$ trees of the ordinary increasing tree family.

\smallskip

Next, we consider star-shaped trees $\tilde{T}$ of size $b+k$ and root degree $k$. We get
\[
w(\tilde{T})=\varphi_k\cdot \psi_1^k.
\]
On the other hand, we can describe all trees $T\in \mathcal{T}^{[1]}$ satisfying $\mathcal{C}(T)=\tilde{T}$. Namely, all such trees can be created by taking an arbitrary tree of size $b$ and attaching $k$ new nodes labelled $k+1,\dots,k+b$ to any of the existing $b$ nodes.
We denote by $v_1,\dots,v_b$ the nodes labelled $1,\dots,b$ in $T$, by $\grad^{+}_{\le b}(v_{m})$ the out-degree of node $v_{m}$ when restricting to labels $\le b$, and by $j_{m}$ the number of nodes with labels $> b$ attached to $v_{m}$. This gives
\[
\sum_{T\in \mathcal{T}^{[1]}\colon \mathcal{C}(T)=\tilde{T}}w(T)
=\phi_0^k\sum_{T\in \mathcal{T}^{[1]}\colon |T|=b}w(T)\sum_{\substack{\sum_{s=1}^{b}j_s=k,\\j_s\ge 0}}\prod_{m=1}^{b}\frac{\phi_{\grad^{+}_{\le b}(v_m)+j_m}}{\phi_{\grad^{+}_{\le b}(v_m)}}\binom{\grad^{+}_{\le b}(v_m)+j_m}{j_m}.
\]
We use $\psi_1=\phi_0=T_1$ and obtain for $\varphi_k$ the expression
\begin{equation}
\label{BINCCondition2}
\varphi_k=\sum_{T\in \mathcal{T}^{[1]}\colon |T|=b}w(T)\sum_{\substack{\sum_{s=1}^{b}j_s=k,\\j_s\ge 0}}\prod_{m=1}^{b}\frac{\phi_{\grad^{+}_{\le b}(v_m)+j_m}}{\phi_{\grad^{+}_{\le b}(v_m)}}\binom{\grad^{+}_{\le b}(v_m)+j_m}{j_m}.
\end{equation}

\medskip

As mentioned earlier, it has been shown in~\cite{PanPro2007} that only three families of increasing trees can be constructed using a stochastic growth process. These are ordinary recursive trees, generalized plane-oriented recursive trees and $d$-ary increasing trees, and they are determined by the degree-weight sequences $\phi_k=\frac1{k!}$, $\phi_k=\binom{k+\alpha-1}{k}$ and $\phi_k=\binom{d}{k}$, $k\ge 0$, respectively. The three weight sequences $(\phi_k)_{k\ge 0}$ together with~\eqref{BINCCondition2} directly lead to the result of Theorem~\ref{BINCprop1}. Below we present the calculations for \Port{} derived from generalized plane-oriented recursive trees with $\phi_k=\binom{k+\alpha-1}{k}$. 
Using 
\[
\frac{\phi_{\grad^{+}_{\le b}(v_m)+j_m}}{\phi_{\grad^{+}_{\le b}(v_m)}}\binom{\grad^{+}_{\le b}(v_m)+j_m}{j_m}=
\binom{\grad^{+}_{\le b}(v_m)+j_m+\alpha-1}{j_{m}}
\]
and 
\begin{align*}
&\sum_{\substack{\sum_{s=1}^{b}j_s=k,\\j_s\ge 0}}\prod_{m=1}^{b}\binom{\grad^{+}_{\le b}(v_m)+j_m+\alpha-1}{j_m}
=\sum_{\substack{\sum_{s=1}^{b}j_s=k,\\j_s\ge 0}}\prod_{m=1}^{b}[t^{j_m}]\frac{1}{(1-t)^{\grad^{+}_{\le b}(v_m)+\alpha}}\\
&\quad=[t^{k}]\frac{1}{(1-t)^{b-1+b\alpha}}=\binom{b+b\alpha+k-2}{k},
\end{align*}
we get indeed the result stated in Theorem~\ref{BINCprop1}:
\begin{equation}
\label{BINCCondition}
\varphi_k=\sum_{T\in \mathcal{T}_1\colon |T|=b}w(T)\cdot \binom{b+b\alpha+k-2}{k}
=T_b\cdot \binom{b+b\alpha+k-2}{k}.
\end{equation}

\smallskip

\begin{example}[Bijection between PORTs and three-bundled bilabelled bucket ordered increasing trees]
\label{BIJex1}
Following the terminology of Janson et al.~\cite{Janson2011}, we call an ordered (bucket) increasing tree $d$-bundled, 
if every node has $d$ positions, with a (possibly empty) sequence of $d$-bundled trees (with disjoint sets of
labels) attached to each position. Equivalently, one may think of each node of the ordered tree as having $d-1$ separation walls, which can be regarded as a special type of edges or half-edges, that separate the subtrees of each node into $d$ bundles. It is known that $d$-bundled increasing trees correspond to ordinary ordered increasing trees with degree-weight generating function $\varphi(t)=1/(1-t)^{d}$; analogous the family $\mathcal{B}_{d}$ of $d$-bundled bucket increasing trees corresponds to the family of bucket ordered increasing trees with $\varphi(t) = 1/(1-t)^{d}$.

Given an ordinary plane-oriented recursive tree, we modify the clustering map $\mathcal{C}$ defined in Algorithm~\ref{cluster} to a map $\mathcal{C}^{\ast}$, such that the corresponding bilabelled bucket ordered increasing tree is three-bundled: when redirecting edges to the buckets, we keep track of their previous ancestor. Given a resulting bucket $v=(\ell_{\min},\ell_{\max})$, edges originally incident to the node with smaller label $\ell_{\min}$ that are lying to the left of the node with larger label $\ell_{\max}$ are grouped into the first bundle, whereas edges incident to $\ell_{\min}$ lying to the right of $\ell_{\max}$ are grouped into the third bundle. All edges incident to the node with larger label $\ell_{\max}$ are put into the second bundle, the center. 

\smallskip

Then, the map $\mathcal{C}_\ast\colon\mathcal{T}^{[1]}\to\mathcal{B}_3$, with bucket size $b=2$ of bucket three-bundled ordered increasing trees, is a bijection. This bijection is illustrated in Figure~\ref{ClusterFig2}. As mentioned before, it also explains the corresponding enumerative result of~\cite{Hwang2016}: the number of three-bundled ordered increasing trees with degree-weight generating function $\varphi(t)=1/(1-t)^{3}$ and $\psi_1=1$, and thus of the corresponding family of increasing diamonds, see Theorem~\ref{bijDiamondBucket}, is given by $(2n-3)!!$.

\begin{figure}[!htb]
\begin{center}
\includegraphics[height=4.5cm]{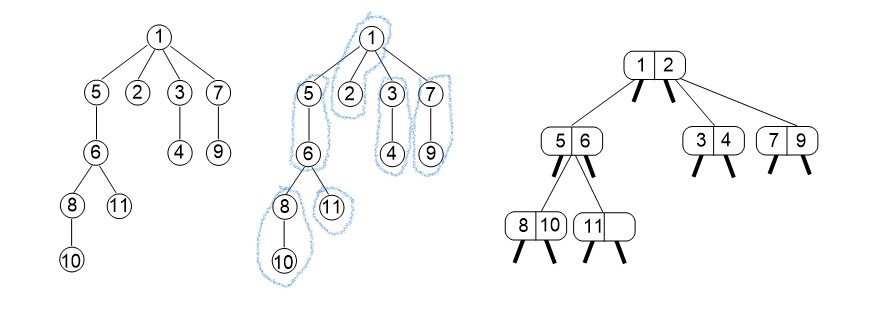}
\end{center}
\caption{The plane-oriented recursive tree $T$ of size eleven from Figure~\ref{ClusterFig1}, the corresponding clustered tree with $b=2$, and the corresponding bilabelled $3$-bundled increasing ordered tree $\tilde{T}=\mathcal{C}^{\ast}(T)$.\label{ClusterFig2} }
\end{figure}

\end{example}

\begin{example}[Bijection between recursive trees and two-bundled bilabelled bucket recursive trees]
\label{BIJex2}
We introduce two-bundled bucket recursive trees $\mathcal{BR}_2$ with bucket size two. Each saturated node has an half-edge that separates its subtrees into two bundles. 
The tree family corresponds to the degree-weight generating function $\varphi(t)=\exp(2t)$ and $\psi_1=1$. Then, we modify the clustering map $\mathcal{C}$ again to a map $\hat{\mathcal{C}}$, such that the corresponding bilabelled bucket recursive tree is two-bundled: when redirecting edges to the buckets, we keep track of their previous ancestor. Given a resulting bucket $v=(\ell_{\min},\ell_{\max})$, all edges originally incident to the smaller label $\ell_{\min}$ are grouped into the first bundle, whereas all edges incident to the larger label $\ell_{\max}$ are put into the second bundle. 

\smallskip

Then, the map $\hat{\mathcal{C}}\colon \mathcal{R} \to \mathcal{BR}_2$ from ordinary recursive trees to two-bundled bilabelled bucket recursive trees is a bijection. 
\end{example}

Similar bijections for bucket size $b=2$ can be obtained also for binary increasing trees and variants.

\section{Increasing diamonds and bucket increasing trees}
\subsection{Symbolic description and bijections}
We consider trees with bucket size $b=2$, degree-weight generating function $\varphi(t)$ and weight $\psi = \psi_{1} = 1$. 
As stated in Example~\ref{example:bilabelled} they allow the symbolic description
\begin{equation}
\label{DiamondEqn1}
\mathcal{T} = \mathcal{Z}^{\Box} +\mathcal{Z}^{\Box} \ast \left(\mathcal{Z}^{\Box} \ast \varphi\big(\mathcal{T}\big)\right).
\end{equation}
The exponential generating function $T(z)$ satisfies the first-order differential equation
\[
    T'(z) = \sqrt{1+2\cdot\Phi(T(z))}, \quad T(0)=0,
\]
with $\Phi(x)=\int_{0}^{x} \varphi(t)dt$. Moreover, $T=T(z)$ is implicitly given via
\[
\int_{0}^{T}\frac{dx}{\sqrt{1+2\cdot\Phi(x)}}=z.
\]
On the other hand, increasing diamonds as proposed in \cite{Hwang2016} are defined by the symbolic equation
\begin{equation}
\label{DiamondEqn2}
\mathcal{F} = \mathcal{Z}^{\Box} +\mathcal{Z}^{\Box} \ast  \varphi\big(\mathcal{F}\big) \ast \mathcal{Z}^{\blacksquare},
\end{equation}
where the latter boxed product constrains the largest label. They constitute a combinatorial family of labelled, directed acyclic graphs (DAGs) with a source and a sink such that the labels along any path are increasing. The functional operation $\varphi$ occurring in \eqref{DiamondEqn2} is here specifying possible degrees respectively their weights, see \cite{Hwang2016} for more details. Figure~\ref{DiamondFig1} illustrates two examples of increasing diamonds.

\begin{figure}[!htb]
\begin{center}
\begin{minipage}{13cm}
\includegraphics[height=5cm]{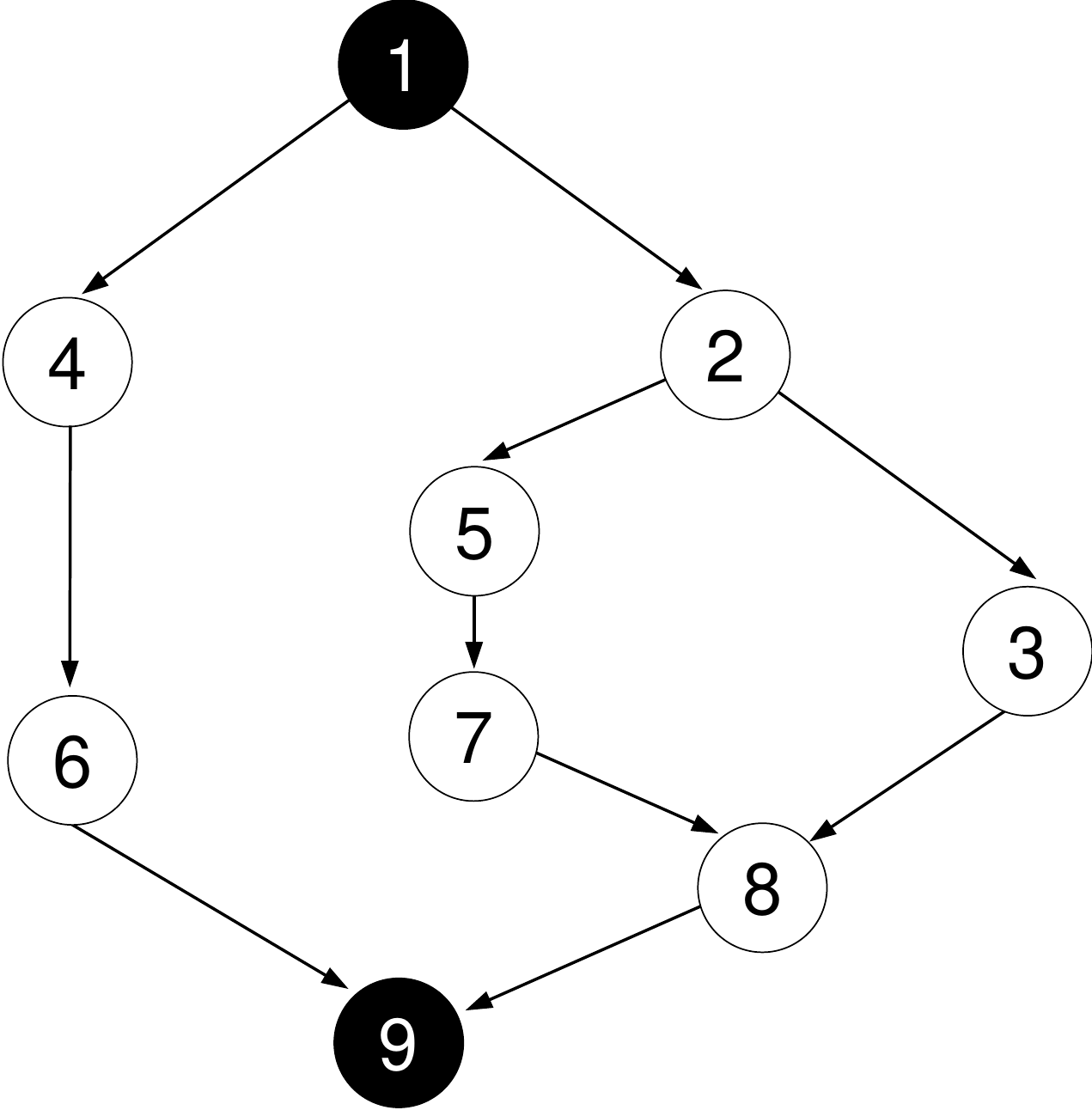}
\hspace{1cm}
\includegraphics[height=5cm]{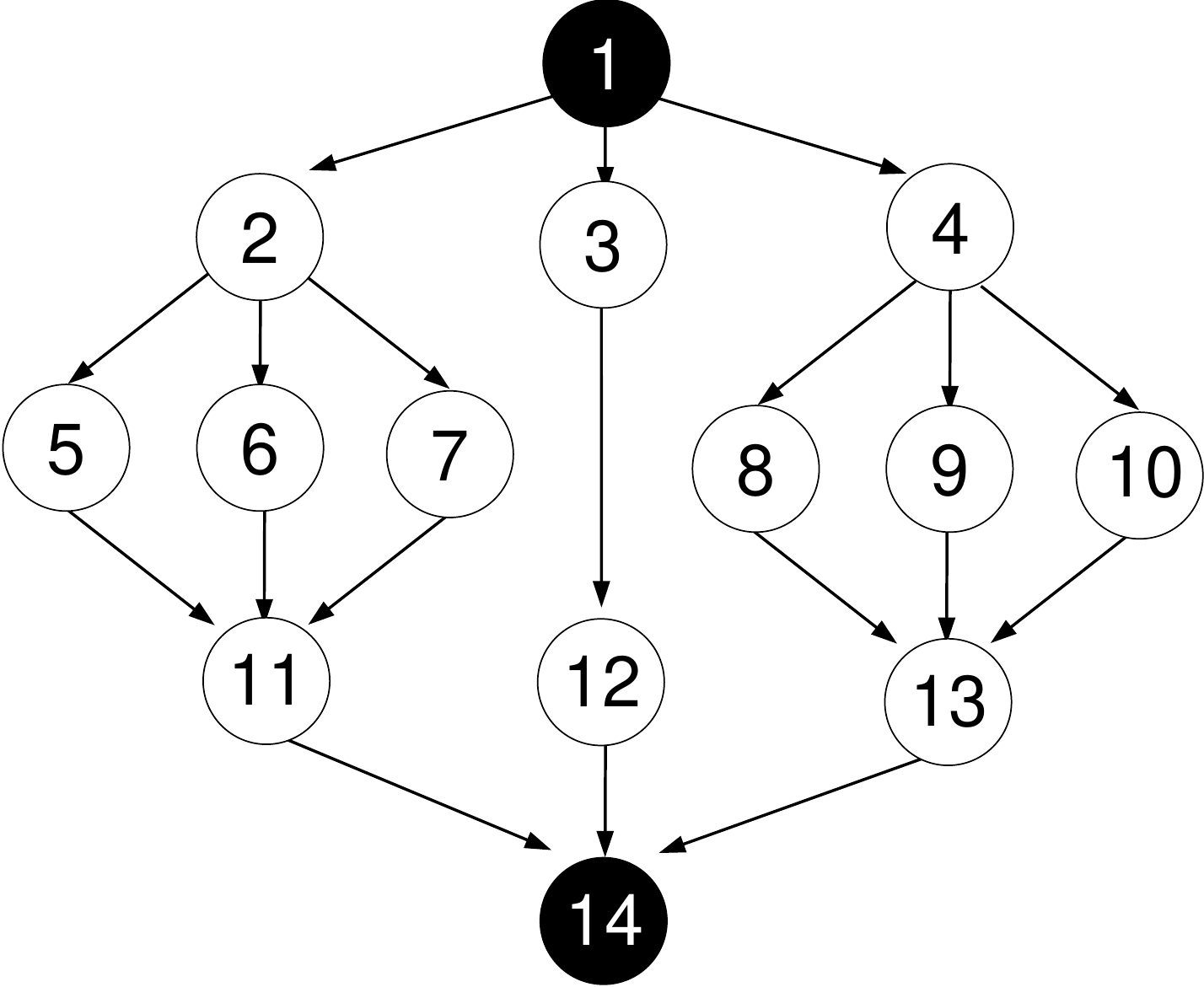}
\end{minipage}
\end{center}
\caption{Two different increasing diamonds of size nine and fourteen, respectively.\label{DiamondFig1} }
\end{figure}

The symbolic description \eqref{DiamondEqn2} of increasing diamonds leads to exactly the same exponential generating function as for bucket increasing trees with bucket size $b=2$ described formally via \eqref{DiamondEqn1}.
Thus, it is natural to ask for a bijection between these combinatorial objects of a given size, where the size of an increasing diamond is given by the number of its nodes. 

Before stating such a bijection we note that it is convenient to think of an increasing diamond $F\in\mathcal{F}$ as having three types of nodes stemming from the recursive combinatorial construction given in \eqref{DiamondEqn2}, which partitions the set of nodes into small nodes, inner nodes and large nodes, i.e., 
\begin{equation*}
   V(F) = V_{S}(F) \: \dot{\cup} \: V_{I}(F) \: \dot{\cup} \: V_{L}(F).
\end{equation*}
Namely, if $F$ has size one its only vertex is assigned to an inner node. Otherwise, the smallest node and the largest node of $F$ are assigned to the respective node type; furthermore by removing these two nodes of $F$ we obtain a (possibly empty) sequence of increasing diamonds $F_{1}, \dots, F_{r}$, and to each of these structures we apply this assignment recursively. 

\begin{theorem}
\label{bijDiamondBucket}
The family of ordered increasing diamonds $\mathcal{F}_n$ of size $n$, with degree-weight generating function $\varphi(t)$,
are in a natural bijection $\mathcal{M}$ with ordered bucket increasing trees $\mathcal{T}_n$ of size $n$ with the same degree-weight generating function, $\mathcal{F}_n\cong \mathcal{T}_n$.
\end{theorem}

\begin{remark}
The bijection $\mathcal{M}$ together with the characterization given in Theorem~\ref{BINCprop1} provides 
a characterization of families $\mathcal{F}$ of increasing diamonds with a total weight $T_n$, resembling the weights of ordinary increasing trees. 
This explains the strikingly simple formulas observed by Bodini et al.~\cite{Hwang2016}.
\end{remark}

\begin{proof}
Our bijection $\mathcal{M} : \mathcal{F}_{n} \to \mathcal{T}_{n}$ uses two steps: first we construct recursively intermediate objects that we call increasing-decreasing bilabelled trees, where each bucket holds the smallest as well as the largest label in its subtree. Thus, when considering a tree $\hat{T} \in \mathcal{ID}$ of this family and restricting to the smaller or larger labels in each bucket, it is an increasing or decreasing tree, respectively.
Then we give a recursive procedure that transforms, by using cyclic permutations of the labels, such trees to bucket increasing trees of the same shape. These recursive procedures are given as Algorithms~\ref{Alg:DiamondToIncdesc}-\ref{Alg:IncdescToBucket}; combining them leads to Algorithm~\ref{diamond} and thus the map $\mathcal{M}$. We further observe that the weight sequence $(\varphi_k)_{k\ge 0}$ is not involved in the bijection, as the weights are directly preserved: if a certain substructure of an increasing diamonds $F$ decomposes into the node with smallest label, the node with largest label and $r \ge 0$ increasing subdiamonds, then the corresponding subtree in the bucket increasing tree $\tilde{T} = \mathcal{M}(F)$ decomposes into the root bucket and exactly $r$ subtrees. We also note that the reverse map $\mathcal{M}^{-1}$ is readily obtained by inverting the recursive procedures in a natural way.
\begin{algorithm}
\caption{DiamondToBucket$(F)$}
\label{diamond}
\begin{algorithmic}[1]
\Input{Increasing Diamond $F\in\mathcal{F}$}
\Result{Bucket size two increasing tree $\tilde{T}\in\mathcal{T}$ of same size and with same label set}
\State $\hat{T} \gets \text{DiamondToIncdesc}(F)$
\State $\tilde{T} \gets \text{IncdescToBucket}(\hat{T})$
\State Return $\tilde{T}$
\end{algorithmic}
\end{algorithm}
\end{proof}

\begin{algorithm}
\caption{DiamondToIncdesc$(F)$}
\label{Alg:DiamondToIncdesc}
\begin{algorithmic}[1]
\Input{Increasing diamond $F \in \mathcal{F}$}
\Result{Increasing-decreasing bilabelled tree $\hat{T} \in \mathcal{ID}$ of same size and with same label set}
\State $\hat{F} \gets F$
\State $n \gets$ size of $\hat{F}$
\If{$n=1$} 
\State form a bucket $\hat{b} = (v_{I} \mid )$ with $v_{I}$ the single node of $\hat{F}$
\State $\hat{T} \gets$ tree consisting of single bucket $\hat{b}$
\State Return $\hat{T}$
\Else 
\State form a bucket $\hat{b} = (v_{S} \mid v_{L})$ with $v_{S}$ and $v_{L}$ the nodes with smallest and largest label of $\hat{F}$
\State $\hat{F}_{1}, \dots, \hat{F}_{r} \gets$ sequence of increasing diamonds obtained by removing $v_{S}$ and $v_{L}$ from $\hat{F}$
\For{$j=1$ to $r$}
\State $\hat{T}_{j} \gets \text{DiamondToIncdesc}(\hat{F}_{j})$
\EndFor
\State $\hat{T} \gets$ tree with root bucket $\hat{b}$ and subtrees $\hat{T}_{1}, \dots, \hat{T}_{r}$ attached to it
\State Return $\hat{T}$
\EndIf
\end{algorithmic}
\end{algorithm}

\begin{algorithm}
\caption{IncdescToBucket$(\hat{T})$}
\label{Alg:IncdescToBucket}
\begin{algorithmic}[1]
\Input{Increasing-decreasing bilabelled tree $\hat{T}\in\mathcal{ID}$}
\Result{Bucket size $2$ increasing tree $\tilde{T}\in\mathcal{T}$ of same size, same shape and with same label set}
\State $\tilde{T} \gets \hat{T}$
\State $n \gets$ size of $\tilde{T}$
\If{$n=1$} 
\State Return $\tilde{T}$
	\Else
\State let $\ell_{1} < \ell_{2} < \dots <\ell_{n}$ be the labels of $\tilde{T}$ in increasing order
\State define $\pi \gets (\ell_{1})(\ell_{2} \: \ell_{3} \: \dots \ell_{n})$ permutation in cycle notation
\State permute labels of $\tilde{T}$ according to permutation $\pi$
\State $\tilde{b} \gets$ root bucket of $\tilde{T}$
\State $\hat{T}_{1}, \hat{T}_{2}, \dots, \hat{T}_{r} \gets$ subtrees of root of $\tilde{T}$
\For{$j=1$ to $r$}
\State $\tilde{T}_{j} \gets \text{IncdescToBucket}(\hat{T}_{j})$
\EndFor
\State $\tilde{T} \gets$ tree with root bucket $\tilde{b}$ and subtrees $\tilde{T}_{1}, \dots, \tilde{T}_{r}$ attached to it
\State Return $\tilde{T}$
\EndIf
\end{algorithmic}
\end{algorithm}

The bijection $\mathcal{M}$ is illustrated in Figures~\ref{Fig:BijectionM_Diamond1}-\ref{Fig:BijectionM_Diamond2} for the increasing diamonds given in Figure~\ref{DiamondFig1}.


\begin{figure}[!htb]
\begin{center}
\begin{minipage}{14cm}
\includegraphics[width=4cm]{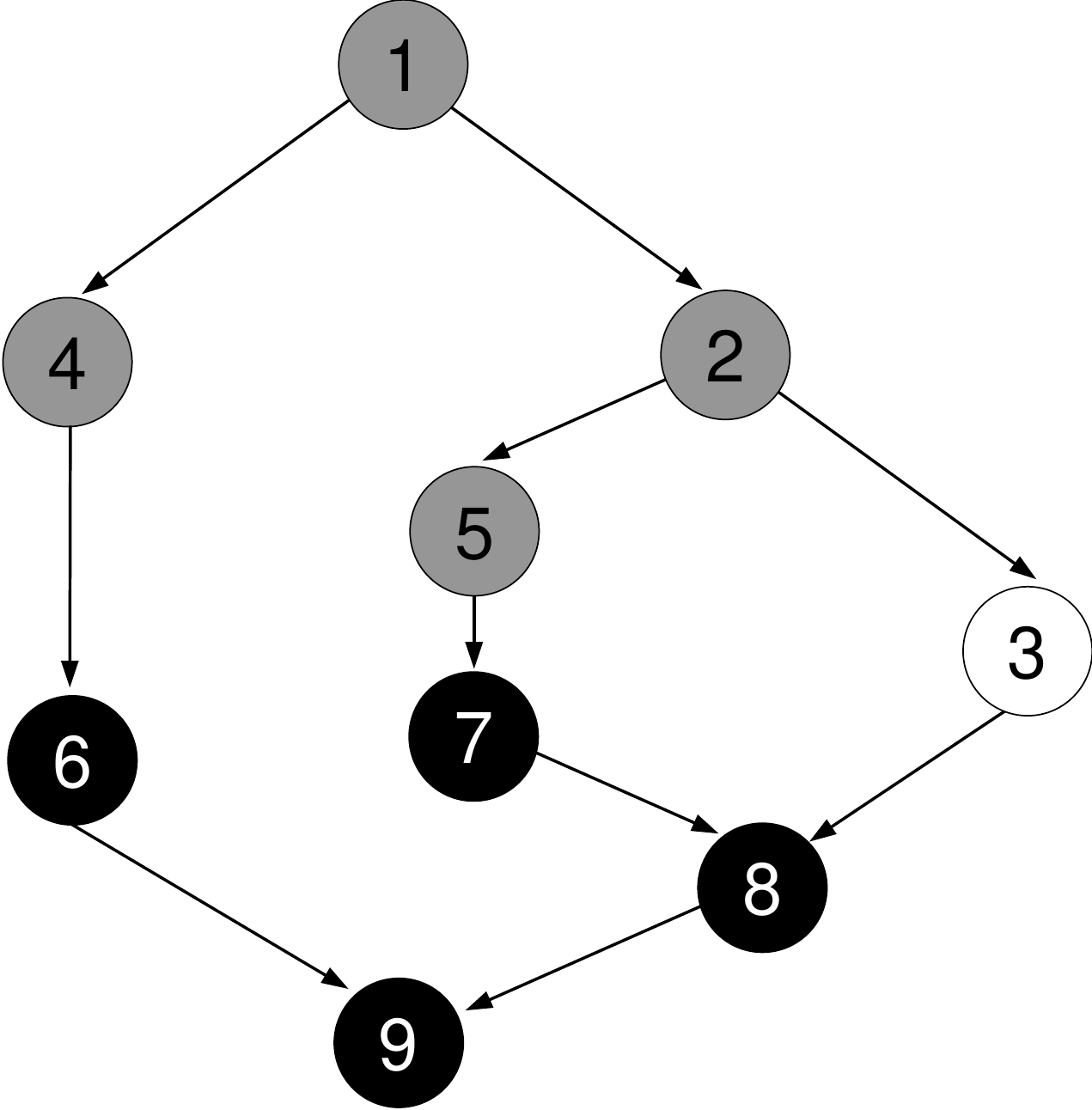}
\hspace{0.1cm}
\includegraphics[width=4.5cm]{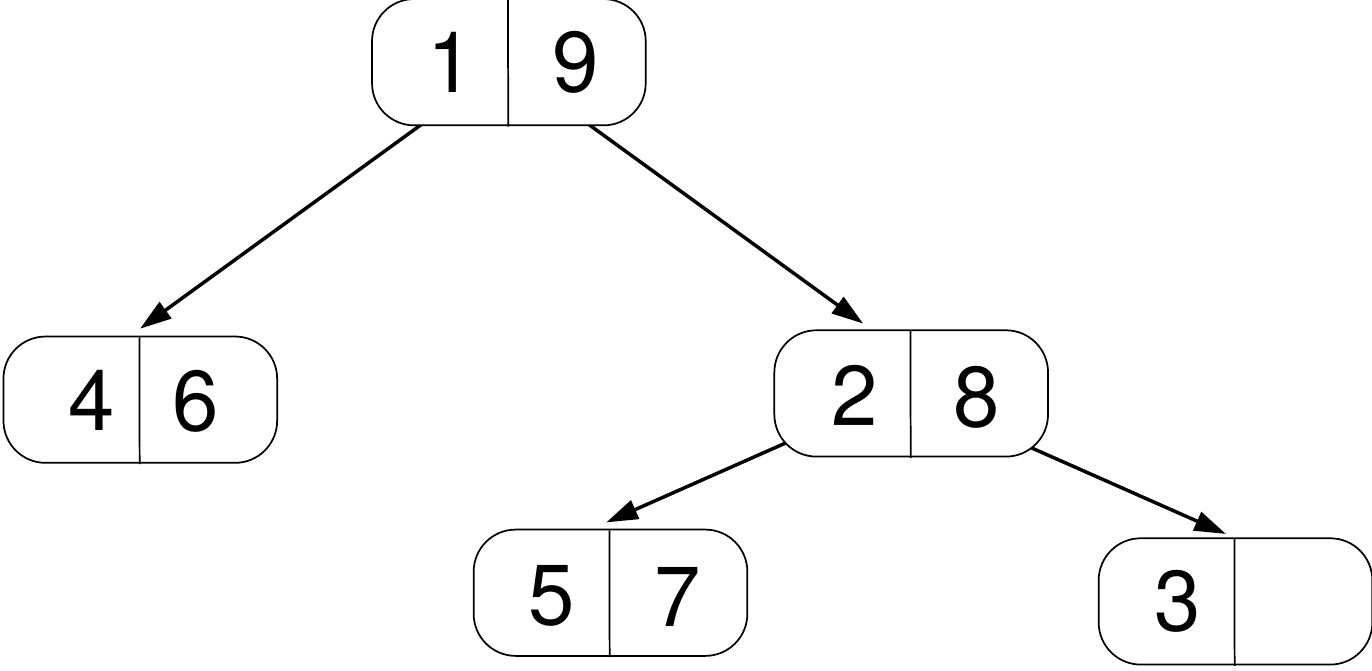}
\hspace{0.1cm}
\includegraphics[width=4.5cm]{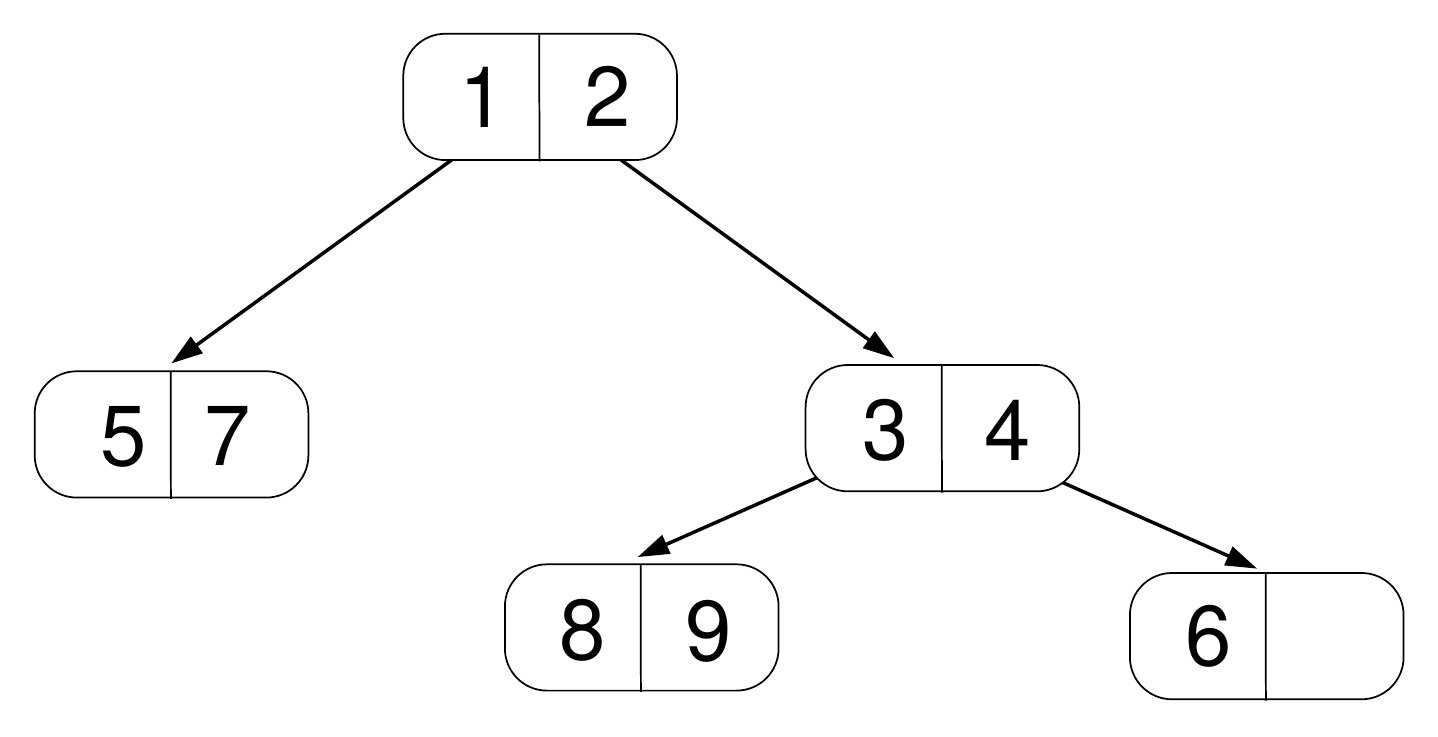}
\end{minipage}
\end{center}
\caption{The increasing diamond of size nine from Figure~\ref{DiamondFig1} with three different label types is mapped to an increasing-decreasing bilabelled tree and then to a bucket increasing tree.\label{Fig:BijectionM_Diamond1}}
\end{figure}

\begin{figure}[!htb]
\begin{center}
\begin{minipage}{14cm}
\includegraphics[width=4cm]{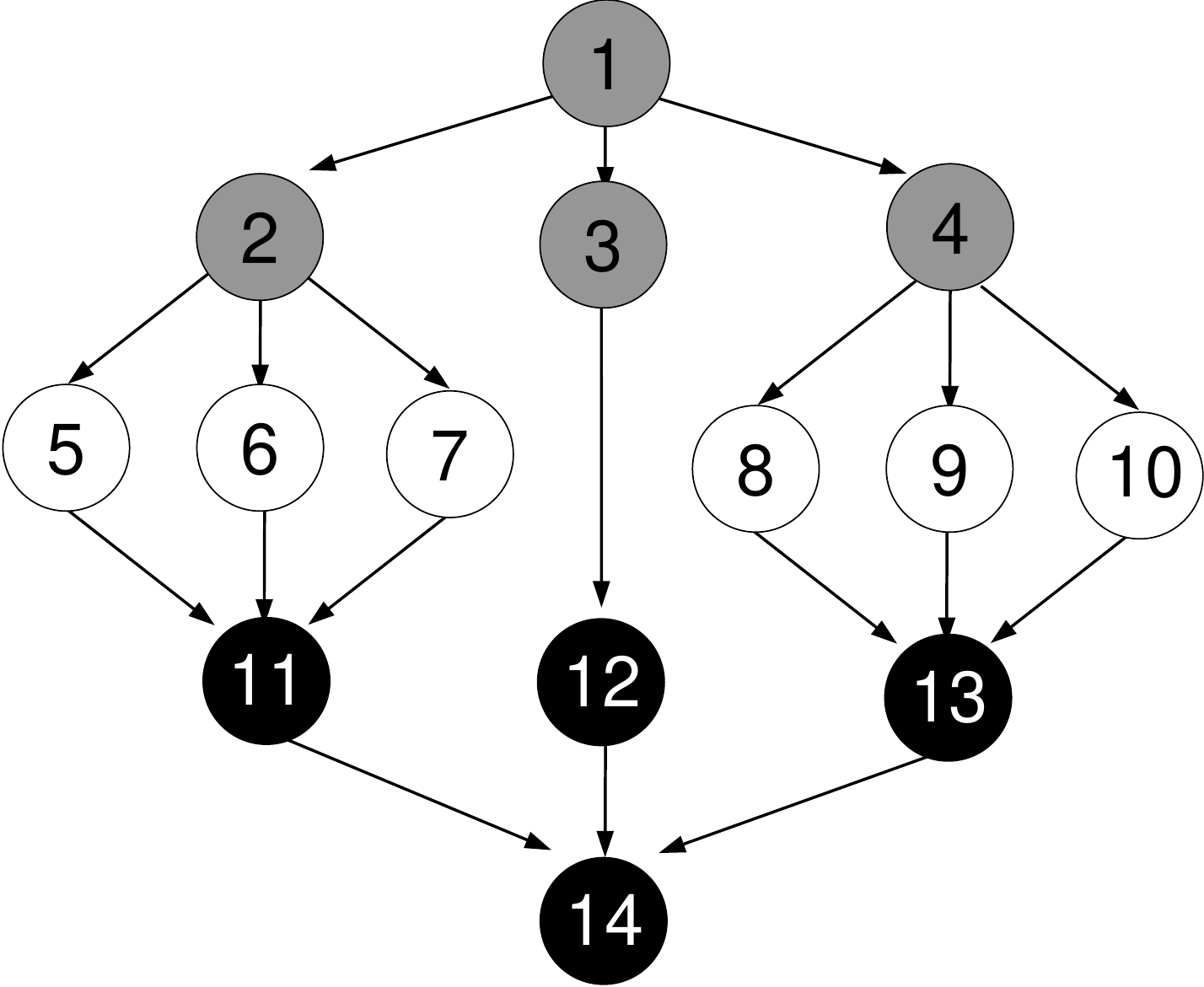}
\hspace{0.1cm}
\includegraphics[width=4.5cm]{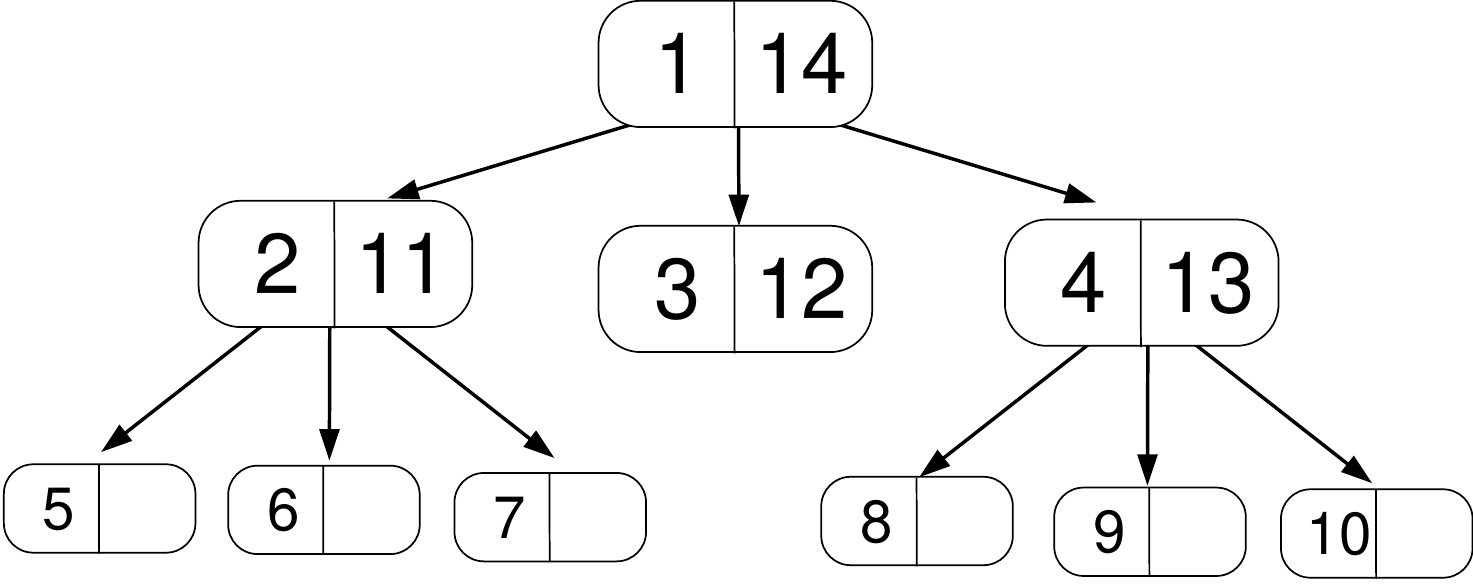}
\hspace{0.1cm}
\includegraphics[width=4.5cm]{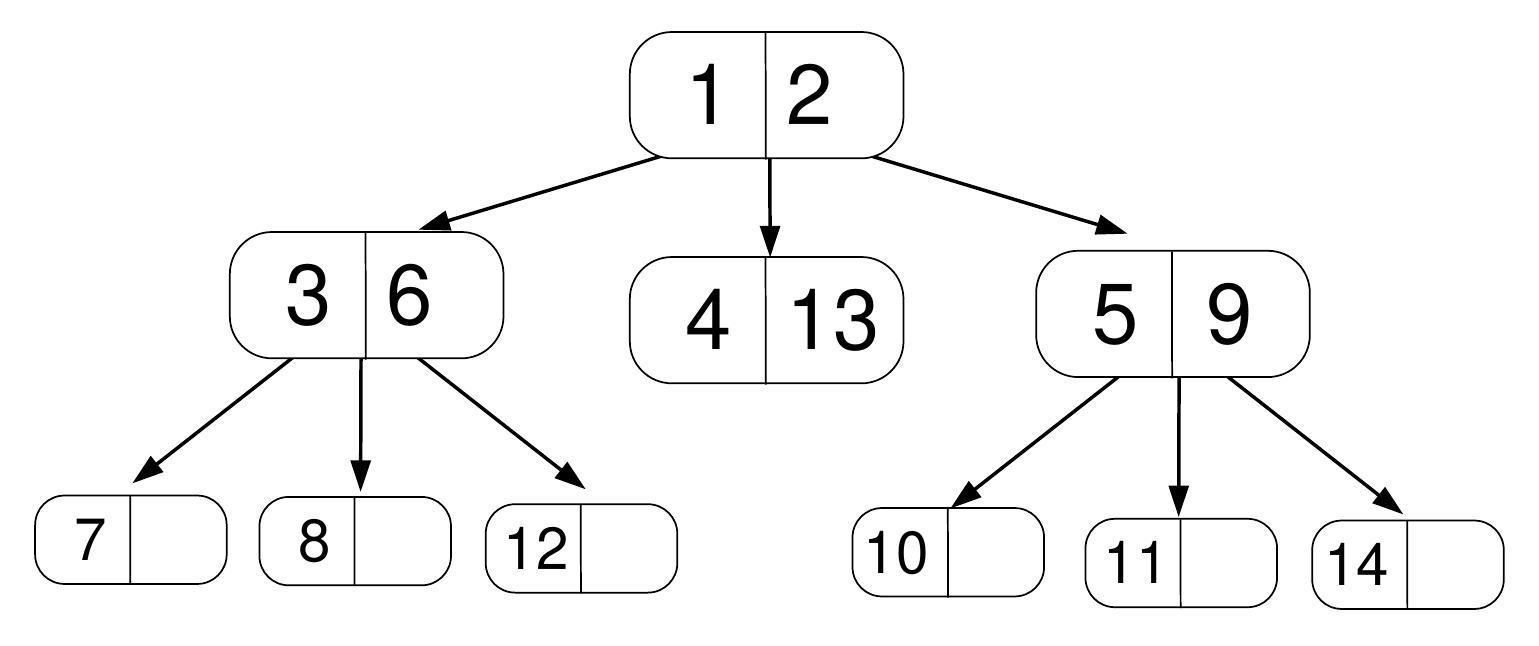}
\end{minipage}
\end{center}
\caption{The increasing diamond from Figure~\ref{DiamondFig1} of size $14$ with three different label types is mapped to an increasing-decreasing bilabelled tree and then to a bucket increasing tree.\label{Fig:BijectionM_Diamond2}}
\end{figure}


The map also gives a correspondence between quantities in increasing diamonds and bucket increasing trees.
Let $I_{n}$ denote the random variable counting the number of inner nodes in a random increasing diamond of size $n$. 
Moreover, let $N_n$ denote the number of nodes with capacity one in a random bucket increasing tree of size $n$. Then we get the following corollary.

\begin{coroll}
Under the bijection $\mathcal{M}$ given in Theorem~\ref{bijDiamondBucket} the number of inner nodes in an increasing diamond $F$ coincides with the number of nodes of capacity one in the corresponding bucket increasing tree $\mathcal{M}(F)$. Consequently, the r.v.\ $I_{n}$ and $N_{n}$ are equal in distribution, $I_{n} \stackrel{(d)}{=} N_{n}$.
\end{coroll}

We note that a refined enumeration of increasing diamonds according to the number of inner nodes is possible in a rather direct way. 
Let 
\[
f(z,u)=\sum_{n\ge 1}\sum_{j=1}^{n}\frac{T_{n,j}}{n!}z^n u^j
=\sum_{n\ge 1}\frac{T_n\cdot \E(u^{I_n})}{n!}z^n
\]
denote the refined generating function
with $f(z,1)=T(z)$.

Then from the symbolic equation~\eqref{DiamondEqn2} and the resulting differential equation for $f(z,u)$ one gets that $f=f(z,u)$ is characterized implicitly via
\[
\int_{0}^{f}\frac{dx}{\sqrt{u^2+2\cdot\Phi(x)}}=z.
\]
Many interesting concrete generating functions and examples can be obtained by specializing $\varphi(t)$; the authors are currently investigating into this matter.

%
%
%
%
%
%
%

\section{Initial bucket size of a specified element\label{BINCsecini}}
We consider now the random variable $K_{n}$, which denotes the size of the bucket containing element $n$ in a random bucket increasing tree (with maximal bucket size $b$) of size $n$. Note that by definition $1\le K_{n} \le b$. As it will turn out later, the precise analysis of $K_n$ is required in order 
to obtain distributional decompositions for two further r.v.\ of interest, the number of descendants $Y_{n,j}$ of label $j$ as well as the out-degree $X_{n,j}$ node $j$. 
We derive the probability mass function of $K_n$ using a generating function approach. Then, we relate the parameter $K_n$ to the distribution of node types in bucket increasing trees.

\subsection{The generating functions approach\label{BINCsecini1}}

In order to study $K_{n}$ for bucket increasing trees we introduce the bivariate generating function
\begin{equation}
   \label{eqnb0}
   N(z,v) := \sum_{n \ge 0} \sum_{m \ge 0} \mathbb{P}\{K_{n+1} = m\} T_{n+1}
   \frac{z^{n}}{n!} v^{m}.
\end{equation}
To establish a functional equation for $N(z,v)$ from the formal recursive equation~\eqref{eqna2} it is now
convenient to think of specifically bicoloured bucket increasing trees,
where the colouring is as follows: exactly one element, namely the element with largest label, is
coloured \emph{red}, and all elements having a label smaller than the red element
are coloured \emph{black}. Let us first assume that the red element of $T$ is not contained in the root node.
Then the red element is located in one of the $r$ subtrees of the
root of $T$; let us assume that it is in the $j$-th subtree. Let us
now consider these $r$ subtrees. After order preserving
relabellings, each subtree $S_{1}, \dots, S_{r}$ is an bucket increasing tree
by itself, where one of the $r$ subtrees contains the red element. Note the obvious fact that the size of the bucket of the red element is the same in $T$ and in the respective subtree $S_{j}$.

We introduce now generating functions, with exponential 
variable $z$, where $z$ marks the black elements, $f(z) = \sum_{n \ge 0} f_{n} \frac{z^n}{n! }$ for sequences $f_{n}$ and $f(z,v)=\sum_{n,m \ge 0} f_{n,m} \frac{z^n }{n!}v^m$ for sequences $f_{n,m}$,
where $v$ counts the initial bucket size of the red element.
With this setting, the total weight of all suitably bicoloured
bucket increasing trees, where the
initial bucket size of the red element is exactly $m$, is given by
$\P\{K_{n+1} = m\} T_{n+1}$, and thus its generating function
is given by
\begin{equation*}
   \sum_{n \ge 0} \sum_{m \ge 0} \mathbb{P}\{K_{n+1}=m\}
   T_{n+1} \frac{z^{n}}{n!} v^{m} = N(z,v),
\end{equation*}
whereas the total weight of suitably monocoloured ordinary bucket increasing trees
is $T_{n}$ and its generating
function is given by
\begin{equation*}
   \sum_{n \ge 1}  T_{n} \frac{z^{n}}{n!} = T(z).
\end{equation*}
The $r-1$ monocoloured trees and the bicoloured bucket tree lead then to the expression
$T(z)^{r-1} \cdot N(z,v)$. Since the red element can be in the first,
second, $\dots$, $r$-th subtree, we additionally get a factor $r$.
Furthermore, the event that the root has out-degree $r$ leads to a factor $\varphi_r$.
Summing over all $r \ge 1$ leads thus to $\sum_{r\ge 1} r \varphi_r T(z)^{r-1} N(z,v) =
\varphi'(T(z)) N(z,v).$
Since the elements labelled by $1, 2, \dots, b$ contained in the root node are all coloured black
(which again means that $b$ elements in a labelled object are fixed),
equation \eqref{eqna2} leads thus to the following differential equation of order $b$ for $N(z,v)$:
\begin{equation}
   \label{eqnb1}
   \frac{\partial^b}{\partial z^b}N(z,v) = \varphi'(T(z)) \cdot N(z,v).
\end{equation}

The cases, where the red element is contained in the root of the tree do not appear explicitly in the
differential equation itself, but will be described by the initial conditions.
Since $\mathbb{P}\{K_{n}=n\}=1$, for $1 \le n \le b$
(if element $n$ is contained in the root node then all elements with a label $\le n$ are also contained in the root node),
we obtain the following initial conditions, for $0 \le \ell \le b-1$:
\begin{align*}
   \left.\frac{\partial^{\ell}}{\partial z^{\ell}} N(z,v)\right|_{z=0} & = \sum_{m \ge 0}
   \mathbb{P}\{K_{\ell+1}=m\} T_{\ell+1} v^{m}=T_{\ell+1} v^{\ell+1},
\end{align*}
with $T_{n} = n! [z^{n}] T(z)$ and $T(z)$ as characterized in Proposition~\ref{BINCprop1} for the particular tree families.

Now we can specify the sequences according to the tree family of interest. Note that for bucket recursive trees 
the initial bucket size of node $n$ was already (implicitly) characterized in~\cite{BucketPanKu2009}, hence we will skip this case. For \Dit{} we obtain the following Cauchy-Euler type differential equation together with the initial conditions for the
bivariate generating function $N(z,v)$:
\begin{equation}
 \label{eqnb3a}
   \frac{\partial^{b}}{\partial z^{b}} N(z,v) = \frac{T_{b+1}}{(1-(d-1)z)^{b}}N(z,v),
   \quad \left.\frac{\partial^{\ell}}{\partial z^{\ell}} N(z,v)\right|_{z=0} = T_{\ell+1}v^{\ell+1},
   \quad \text{for} \enspace 0 \le \ell \le b-1.
\end{equation}
For \Port{} we obtain a very similar 
Cauchy-Euler type differential equation together with the initial conditions for the
bivariate generating function $N(z,v)$:
\begin{equation}
\label{eqnb3b}
   \frac{\partial^{b}}{\partial z^{b}} N(z,v) = \frac{T_{b+1}}{(1-(\alpha+1)z)^{b}}N(z,v),
   \quad \left.\frac{\partial^{\ell}}{\partial z^{\ell}} N(z,v)\right|_{z=0} = T_{\ell+1}v^{\ell+1},
   \quad \text{for} \enspace 0 \le \ell \le b-1.
\end{equation}

\subsection{The distribution of the initial bucket size}\label{ssec52}
In order to obtain the exact distribution of the r.v.\ $K_{n}$ we will give the exact solution of the
homogeneous differential equations \eqref{eqnb3a}, \eqref{eqnb3b}, which are of Cauchy-Euler-type. Plugging in the Ansatz $N(z,v) = \frac{1}{(1-(d-1)z)^{\lambda}}$ for 
\Dit, and $N(z,v) = \frac{1}{(1-(\alpha+1)z)^{\lambda}}$ for \Port,
with unspecified $\lambda$, into equations \eqref{eqnb3a}, \eqref{eqnb3b}
leads to the indicial equation
\begin{equation}
   \label{eqnb4}
   \lambda^{\overline{b}} =
   \begin{cases} 
   \displaystyle{\frac{T_{b+1}}{(d-1)^{b}}},\\[0.2cm]
   \displaystyle{\frac{T_{b+1}}{(\alpha+1)^{b}}},
   \end{cases}
   \quad \text{or equivalently} \quad
   \binom{\lambda+b-1}{b} =
   \begin{cases} 
   \displaystyle{\binom{b+\frac{1}{d-1}}{b}},\quad\dit,\\[0.2cm]
   \displaystyle{\binom{b-\frac{1}{\alpha+1}}{b}},\quad\port.
   \end{cases}
\end{equation}
Similar equations have been studied in \cite{MahSmy1995} for bucket recursive trees. 
It is convenient to give a unified analysis of the indicial equations 
for all three models.
Let 
\begin{equation}
\label{cc}
\quad \kappa=
\begin{cases}
      \displaystyle{0}, & \quad  \text{bucket recursive trees}, \\[3mm]
      \displaystyle{\frac{1}{d-1}}, & \quad \dit, \\[3mm]
      \displaystyle{-\frac{1}{\alpha+1}}, & \quad \port.
   \end{cases}
\end{equation}
Then, the indicial equations can be written in a unified way:
\begin{equation}
\label{indicial}
 \lambda^{\overline{b}} = \fallfak{(b+\kappa)}{b}.
\end{equation}

	%
	%

Equations of this or similar kind have been treated in~\cite{MahSmy1995,BucketPanKu2009,KubPan2019}. It follows from these considerations that, for $\kappa$ given in \eqref{cc}, all solutions $\lambda_{1}, \lambda_{2}, \dots, \lambda_{b}$ of \eqref{indicial} are simple, and when arranging them in descending order of real parts it further holds $1+\frac{1}{d-1}=\lambda_{1} > \Re(\lambda_{2}) \ge \Re(\lambda_{3}) \ge \cdots
\ge \Re(\lambda_{b})$ for \Dit{} and $1-\frac{1}{\alpha+1}=\lambda_{1} > \Re(\lambda_{2}) \ge \Re(\lambda_{3}) \ge \cdots
\ge \Re(\lambda_{b})$ for \Port.
Thus the general solutions of \eqref{eqnb3a}, \eqref{eqnb3b} are given by
\begin{equation}
   \label{eqnb5}
   N(z,v) = 
   \begin{cases}
   \displaystyle{\sum_{i=1}^{b} \frac{\beta_{i}(v)}{(1-(d-1)z)^{\lambda_{i}}}},\quad\dit,\\[0.2cm]
   \displaystyle{\sum_{i=1}^{b} \frac{\beta_{i}(v)}{(1-(\alpha+1)z)^{\lambda_{i}}}},\quad\port,
   \end{cases}
\end{equation}
with certain functions $\beta_{i}(u,v)$, which are specified by the initial conditions as given in \eqref{eqnb3a}, \eqref{eqnb3b}.
When these initial conditions are plugged into \eqref{eqnb5} this leads to the following system of linear
equations for the unknown functions $\beta_{i}(v)$, $1 \le i \le b$:
\begin{equation*}
   \sum_{i=1}^{b} \lambda_{i}^{\overline{\ell}} \beta_{i}(v)=
   \begin{cases}
   \displaystyle{v^{\ell+1}l!\binom{\ell+\frac{1}{d-1}}{\ell}},\quad\dit,\\[0.2cm]
   \displaystyle{v^{\ell+1}l!\binom{\ell-\frac{1}{\alpha+1}}{\ell}},\quad\port.
   \end{cases}
\end{equation*}
Using the abbreviations
\begin{equation}
   \label{eqnb20}
   s_{\ell} := s_{\ell}(v) := 
   \begin{cases}
   \displaystyle{v^{\ell+1}\binom{\ell+\frac{1}{d-1}}{\ell}},\quad\dit,\\[0.2cm]
   \displaystyle{v^{\ell+1}\binom{\ell-\frac{1}{\alpha+1}}{\ell}},\quad\port,
   \end{cases}
\end{equation}
we obtain the following system of linear equations for the unknown $\beta_{i}=\beta_i(v)$, $1 \le i \le b$:
\begin{equation}
   \label{eqnb6}
   \sum_{i=1}^{b} \binom{\lambda_{i}+\ell-1}{\ell} \beta_{i} = s_{\ell},
   \quad \text{for} \enspace 0 \le \ell \le b-1.
\end{equation}
To obtain the solution of \eqref{eqnb6} we can use results and respective computations given in \cite[equations $(22)$ and $(30)$]{BucketPanKu2009}, slightly adapted to $\kappa$ occurring in the indicial equation \eqref{indicial} of the tree families considered.
\begin{lemma}[\cite{BucketPanKu2009}]\label{lem1}
For given numbers $\lambda_i$, with $1\le i\le b$, specified as the solutions of equation~\eqref{indicial},
and numbers $s_{\ell}$, $0\le \ell\le b-1$, 
the system of linear equations with unknowns $\beta_i$, 
\begin{equation*}
\sum_{i=1}^{b} \binom{\lambda_{i}+\ell-1}{\ell-1}\beta_{i}=s_{\ell},\quad 0\le \ell\le b-1,
\end{equation*}
has the exact solution 
\begin{equation*}
\beta_i=
\begin{cases}
\displaystyle{\sum_{r=0}^{b-1}s_r \frac{\binom{\lambda_i+b-1}{b-r-1}}{\binom{b}{r}(b-r)\binom{b+\frac{1}{d-1}}{b}(H_{\lambda_i+b-1}-H_{\lambda_i-1})}},\quad 1\le i\le b,\quad\dit,\\[0.3cm]
\displaystyle{\sum_{r=0}^{b-1}s_r \frac{\binom{\lambda_i+b-1}{b-r-1}}{\binom{b}{r}(b-r)\binom{b-\frac{1}{\alpha+1}}{b}(H_{\lambda_i+b-1}-H_{\lambda_i-1})}},\quad 1\le i\le b,\quad\port.
\end{cases}
\end{equation*}
\end{lemma}
An application of Lemma~\ref{lem1} immediately gives the values $\beta_i(v)$ with respect to the initial conditions \eqref{eqnb20}. Thus extracting coefficients yields
\begin{equation*}
[z^{n-1}v^m]N(z,v)=\P\{K_{n}=m\}\frac{T_{n}}{(n-1)!}=
\begin{cases}
   \displaystyle{\sum_{i=1}^{b}\binom{\lambda_i+n-2}{n-1}(d-1)^{n-1}[v^m]\beta_i(v)},\quad\dit,\\[0.3cm]
   \displaystyle{\sum_{i=1}^{b}\binom{\lambda_i+n-2}{n-1}(\alpha+1)^{n-1}[v^m]\beta_i(v)},\quad\port,
   \end{cases}
\end{equation*}
and by specifying $T_{n}$ as given in Theorem~\ref{BINCprop1} and $s_{\ell}(v)$ as given in \eqref{eqnb20} we obtain the exact distribution of $K_{n}$ stated in the following theorem.


\begin{theorem}
The probability mass function of the random variable $K_n$ counting the initial bucket size of node $n$ in a random bucket tree of size $n$ is given by the following closed formula.
\begin{equation*}
\P\{K_n=m\}=
\begin{cases}
\displaystyle{\sum_{i=1}^{b}\frac{\binom{\lambda_i+n-2}{n-1}\binom{\lambda_i+b-1}{b-m}\binom{m-1+\frac{1}{d-1}}{m-1}}{\binom{n-1+\frac{1}{d-1}}{n-1}\binom{b}{m-1}(b-m+1)\binom{b+\frac{1}{d-1}}{b}(H_{\lambda_i+b-1}-H_{\lambda_i-1})}},\quad\text{for }\dit,\\[0.3cm]
\displaystyle{\sum_{i=1}^{b}\frac{\binom{\lambda_i+n-2}{n-1}\binom{\lambda_i+b-1}{b-m}\binom{m-1-\frac{1}{\alpha+1}}{m-1}}{\binom{n-1-\frac{1}{\alpha+1}}{n-1}\binom{b}{m-1}(b-m+1)\binom{b-\frac{1}{\alpha+1}}{b}(H_{\lambda_i+b-1}-H_{\lambda_i-1})}},\quad\text{for }\port,
\end{cases}
\end{equation*}
for $1\le m\le b$. 

\smallskip

For $n$ tending to infinity the random variable $K_n$ converges in distribution to a limit $K$, whose discrete distribution is given as follows.
\begin{equation*}
\P\{K=m\}=
\begin{cases}
\displaystyle{\frac{\binom{b+\frac{1}{d-1}}{b-m}\binom{m-1+\frac{1}{d-1}}{m-1}}{\binom{b}{m-1}(b-m+1)\binom{b+\frac{1}{d-1}}{b}(H_{b+\frac{1}{d-1}}-H_{\frac{1}{d-1}})}},\quad\dit,\\[0.3cm]
\displaystyle{\frac{\binom{b-\frac{1}{\alpha-1}}{b-m}\binom{m-1-\frac{1}{\alpha+1}}{m-1}}{\binom{b}{m-1}(b-m+1)\binom{b-\frac{1}{\alpha+1}}{b}(H_{b-\frac{1}{\alpha+1}}-H_{-\frac{1}{\alpha+1}})}},\quad\port,
\end{cases}
\end{equation*}
for $1\le m \le b$.
\end{theorem}

\begin{remark}
The corresponding results for bucket recursive trees were already derived in~\cite{BucketPanKu2009} (although not stated explicitly):
\begin{equation*}
   \mathbb{P}\{K_{n}=m\} = \sum_{i=1}^{b} \frac{\binom{\lambda_{i}+b-1}{b-m} \binom{\lambda_{i}+n-2}{n-1}}
   {\binom{b}{m-1} (b-m+1) (H_{\lambda_{i}+b-1} - H_{\lambda_{i}-1})}, 
\end{equation*}
for $1\le m\le b$.  
Furthermore, for bucket recursive trees the random variable $K_n$ converges, for $n \to \infty$, in distribution to a Zipf-distributed limit $K$:
\[
\P\{K=m\}=\frac{1}{m H_b},\quad 1\le m \le b.
\] 
\end{remark}

\begin{proof} 
It remains to show the stated limiting distribution results.
We apply Stirling's formula for the Gamma-function 
\[
    \Gamma(z) = \Bigl(\frac{z}{e}\Bigr)^{z}\frac{\sqrt{2\pi }}{\sqrt{z}}%
    \Bigl(1+\mathcal{O}\big(\frac{1}{z}\big)\Bigr),
\]
and get
\[
 \binom{n-1+\frac{1}{d-1}}{n-1}= 
\frac{\Gamma(n+\frac{1}{d-1})}{\Gamma(1+\frac{1}{d-1})\Gamma(n)}=
\frac{n^{\frac{1}{d-1}}}{\Gamma(1+\frac{1}{d-1})}\big(1+\mathcal{O}(n^{-1})\big),
\]
as well as
\[
 \binom{n-1-\frac{1}{\alpha+1}}{n-1}= 
\frac{\Gamma(n-\frac{1}{\alpha+1})}{\Gamma(1-\frac{1}{\alpha+1})\Gamma(n)}
=\frac{n^{-\frac{1}{\alpha+1}}}{\Gamma(1-\frac{1}{\alpha+1})}\big(1+\mathcal{O}(n^{-1})\big).
\]
Concerning the roots of the indicial equations recall from remarks stated past equation \eqref{indicial} that 
\[
\lambda_1=
\begin{cases}
1+\frac{1}{d-1},\quad\dit,\\
1-\frac{1}{\alpha+1},\quad\port.
\end{cases}
\]
For \Dit{} we obtain the following asymptotic expansions:
\begin{align*}
   \binom{n-2+\lambda_{i}}{n-1} & = \frac{n^{\lambda_{i}-1}}{\Gamma(\lambda_{i})} \big(1+\mathcal{O}(n^{-1})\big) =
   \begin{cases}
      \frac{n^{\frac{1}{d-1}}}{\Gamma(1+\frac{1}{d-1})}\big(1+\mathcal{O}(n^{-1})\big), & \quad i=1, \\
      \mathcal{O}(n^{\Re \lambda_{i}-1}), & \quad 2 \le i \le b.
   \end{cases} 
\end{align*}
For \Port{} we have the corresponding results
\begin{align*}
   \binom{n-2+\lambda_{i}}{n-1}  =
   \begin{cases}
      \frac{n^{-\frac{1}{\alpha+1}}}{\Gamma(1-\frac{1}{\alpha+1})}\big(1+\mathcal{O}(n^{-1})\big), & \quad i=1, \\
      \mathcal{O}(n^{\Re \lambda_{i}-1}), & \quad 2 \le i \le b.
   \end{cases} 
\end{align*}
Thus, for $n\to\infty$, the dominant contribution in the asymptotic expansions of $\P\{K_n=m\}$ and the finite sum stems from the index $i=1$ and $\lambda_1$, leading to the stated result.
\end{proof}

\subsection{Relation to node types in bucket increasing trees}
Let $N_{n,j}$ denote the random variable counting the number of nodes with capacity $c(v)=j$, $1\le j \le b$. 
Furthermore, let $\bN_n$ denote the random vector $(N_{n,1},\dots,N_{n,b})$. Mahmoud and Smythe~\cite{MahSmy1995} considered bucket recursive trees. 
They proved a multivariate central limit theorem for $\bN_n$ for trees with bucket size $b \le 26$. For trees with $b > 26$ a phase change in the limiting distribution of $\bN_n$ was detected and the central limit theorem does not hold anymore. In the following we are going to analyze the limiting distribution of the random vector $\bN_n$ for \Dit\ and \Port. We also summarize the main results for bucket recursive trees.
There is a close connection between the random variables $K_n$ and $\bN_n$. The distribution of the initial bucket size $K_{n+1}$ depends on the different node types present at time $n$, $n\ge 1$.
Let $v_{k}=v_{n,m,k}$ denote the buckets of capacity $m$ contributing to $N_{n,m}$, with $1\le m\le b$ and $1\le k\le N_{n,j}$.
Then, by definition of the growth processes
\[
\P\{K_{n+1}=1 \mid \bN_n\}=\sum_{k=1}^{N_{n,b}}\P\{n+1<_t v_{n,b,k}\}
\]
and, for $2\le m\le b$,
\[
\P\{K_{n+1}=m \mid \bN_n\}=\sum_{k=1}^{N_{n,m-1}}\P\{n+1<_t v_{n,m-1,k}\}.
\]
Of course, for each $k$, the probabilities $\P\{n+1<_t v_{n,b,k}\}$ coincide and are given according to Definitions~\ref{def0}-\ref{def2}:
\[
\P\{n+1<_t v_{n,m,k}\}
=
\begin{cases}
\frac{m}{n}, & \Rec,\\
\frac{(d-1)m+1-\grad^{+}(v)}{(d-1)n+1}, & \dit,\\
\frac{\grad^{+}(v)+(\alpha+1)m-1}{(\alpha+1)n-1}, & \port.\\
\end{cases}
\]
Note that for $1\le m\le b-1$ the individual nodes $v_{k}=v_{n,m,k}$ are unsaturated, such that $\grad^+(v_{n,m,k})=0$.
Consequently, summation leads to the following result.

\begin{prop}
The random variable $K_n$ counting the initial size of the bucket containing label $n$ is related to 
the number of nodes $\bN_n=(N_{n,1},\dots,N_{n,b})$ with respective capacities as follows.
For $2\le m\le b$ it holds
\[
\P\{K_{n+1}=m\}=\E(N_{n,m-1})
\cdot 
\begin{cases}
\frac{m-1}{n}, & \Rec,\\
\frac{(d-1)(m-1)+1}{(d-1)n+1}, & \dit,\\
\frac{(\alpha+1)(m-1)-1}{(\alpha+1)n-1}, & \port.
\end{cases}
\]
and 
\[
\P\{K_{n+1}=1\}=
\begin{cases}
\E(N_{n,b})\cdot \frac{b}{n}, & \Rec,\\
\E(N_{n,b})\cdot \frac{(d-1)b+1}{(d-1)n+1}+\frac{1-\sum_{j=1}^b\E(N_{n,j})}{(d-1)b+1}, & \dit,\\
\E(N_{n,b})\cdot \frac{(\alpha+1)b-1}{(\alpha+1)n-1}+\frac{-1+\sum_{j=1}^b\E(N_{n,j})}{(\alpha+1)n-1}, & \port.
\end{cases}
\]
\end{prop}

In the following we consider generalized P\'olya-Eggenberger urn models (see, e.g., \cite{Janson2004}) with $b$ different types of balls. 
For bucket recursive trees Mahmoud and Smythe~\cite{MahSmy1995} studied the following urn model.

\begin{urn}[Bucket recursive trees]
Consider a balanced urn with balls of $b$ colors and let $Q_{n,m}$ denote the number of balls of type $m$, $1\le m\le b$, in the urn after $n$ draws. $(Q_{n,1},\dots,Q_{n,b})$ denotes the corresponding random vector at time $n$. At each time step, draw one ball at random from the urn, observe its color, and add balls according to the ball replacement matrix
\[
M=
\left(
\begin{matrix}
-1 & 2 & 0 & \dots & 0\\
0 & -2 & 3 & \dots & 0\\
\vdots & & \ddots & & \vdots\\
0 & \dots &0& -(b-1) &b \\
1 & 0 & 0 &\dots  &0 \\
\end{matrix}
\right)
\]
and initial composition a single ball of type one. Then, 
the random variables $Q_{n,m}$ are related to the node types via
\[
N_{n,m} = \frac{Q_{n,m}}{m},\quad 1\le m\le b.
\]
\end{urn}
The characteristic polynomial of $M$ is given 
\begin{equation}
\label{charPol1}
\chi_M(\lambda)=\det(M-\lambda I)=(-1)^b(\auffak{\lambda}{b}-b!).
\end{equation}
For $b\le 26$ and $b > 26$, respectively, there occurs a phase change: the limit law changes from normal to non-normal, which is due to the structure of the characteristic polynomial 
and the general results revealed in \cite{Janson2004}. We assume that the eigenvalues $\lambda_1,\dots,\lambda_b$ are indexed according to their real parts
\[
1=\Re \lambda_1 \ge \Re \lambda_2 \dots \ge \Re \lambda_b.
\]
If the real part of $\lambda_2$ exceeds $\frac12$ then the limit law changes. In the following we denote the random vector of ball types as $\bQ_{n}=(Q_{n,1},\dots,Q_{n,b})$.
\begin{theorem}[Mahmoud and Smythe~\cite{MahSmy1995}]
For $b\le 26$ the limit law of $\frac{1}{\sqrt{n}}\big(\bQ_n-\E(\bQ_n)\big)$ is asymptotically normal. 
For $b>26$ there is no normal limit law (under the same normalization).
\end{theorem}

Here we introduce two new urn models.

\begin{urn}[\Port]
Consider the urn with ball replacement matrix
\[
M=
\left(
\begin{matrix}
-\alpha & 2\alpha +1  & 0 & \dots & 0\\
0 & -(2\alpha +1) & 3\alpha +2 & \dots & 0\\
\vdots & & \ddots & & \vdots\\
0 & \dots &0& -((b-1)\alpha +b-2) &b\alpha +b-1 \\
\alpha & 0 & 0 &\dots  &1 \\
\end{matrix}
\right)
\]
and initial composition $\alpha$ balls of type one. Then, 
the random variables $Q_{n,m}$ are related to the node types by
\[
N_{n,m} = \frac{Q_{n,m}}{m\alpha +m-1},\quad 1\le m\le b.
\]
\end{urn}
The characteristic polynomial of $M$ is given 
\begin{equation}
\label{charPol2}
\chi_M(\lambda)=\det(M-\lambda I)=(-1)^b(\alpha+1)^b\Big(\auffak{\big(\frac{\lambda-1}{\alpha+1}\big)}{b}-\auffak{\big(\frac{\alpha}{\alpha+1}\big)}{b}\Big).
\end{equation}

\begin{urn}[\Dit]
Consider the urn  with ball replacement matrix
\[
M=
\left(
\begin{matrix}
-d & 2d-1  & 0 & \dots & 0\\
0 & -(2d-1) & 3d-2 & \dots & 0\\
\vdots & & \ddots & & \vdots\\
0 & \dots &0& -((b-1)d -(b-2)) &bd-(b-1) \\
d & 0 & 0 &\dots  &-1 \\
\end{matrix}
\right)
\]
and initial composition $d$ balls of type one. Then, 
the random variables $Q_{n,m}$ are related to the node types by
\[
N_{n,m} = \frac{Q_{n,m}}{md -(m-1)},\quad 1\le m\le b.
\]
\end{urn}

The characteristic polynomial of $M$ is given 
\begin{equation}
\label{charPol3}
\chi_M(\lambda)=\det(M-\lambda I)=(-1)^b(d-1)^b\Big(\auffak{\big(\frac{\lambda+1}{d-1}\big)}{b}-\auffak{\big(\frac{d}{d-1}\big)}{b}\Big).
\end{equation}

\smallskip

As indicated by the connection between $K_n$ and $\bN_n$, there is also a close connection between the characteristic polynomials~\eqref{charPol1},~\eqref{charPol2},~\eqref{charPol3} and the indicial equation~\eqref{indicial}:
\[
 \lambda^{\overline{b}} = \fallfak{(b+\kappa)}{b}.
\]
In particular, there is an affine transformation between these two polynomials. 
We note in passing that the general theorems of Janson~\cite{Janson2004} and M\"uller~\cite{Mueller2016} allow to describe a phase change in the limit laws for $\mathbf{N}_n$ similar to the results
of Mahmoud and Smythe~\cite{MahSmy1995}.

\section{Applications}
In the following we present a few applications of the results obtained in Section~\ref{BINCsecini}. The stochastic growth rule discussed in Subsection~\ref{secDESCsub1} and the analysis of the initial bucket size $K_n$ can be used to analyze several parameters. 
We consider in the following the random variable $Y_{n,j}$, which counts the number of descendants of element $j$, i.e., the total number of elements with a label greater or equal $j$ contained in the subtree rooted with the bucket containing element $j$, in a random bucket increasing tree (with maximal bucket size $b$) of size $n$.
For this random variable we provide the exact distribution, as well as a decomposition of the random variable of interest in terms of the initial bucket size $K_n$.
Then we apply our results to the root degree as well as the out-degree $X_{n,j}$ of the node containing element $j$. There we also present a decomposition of the random variable of interest, complementing earlier results. 

\subsection{Descendants in bucket increasing trees} 
In order to avoid degeneracy we assume that $j\ge b+1$ (for $1\le j\le b$ it holds $Y_{n,j}=n+1-j$). Explicit results for the probability mass functions and the moments of both r.v.\ $Y_{n,j}$ and $X_{n,j}$ can be obtained in principle by purely combinatorial means and a generating functions approach similar to the parameter initial bucket size treated above. However, here we take a different point of view utilizing the previous results for the initial bucket size to provide concise decompositions of the random variables based on results already known in the literature. Such decompositions readily lead to limit laws and seem to be more difficult to obtain by using a purely analytic combinatorial approach.

\smallskip

In order to obtain to analyze $Y_{n,j}$, we introduce a refinement of this r.v.: let $Y_{n,\ell,j}=Y_{n,j}\mid K_j=\ell$ denote $Y_{n,j}$ conditioned on the event $K_j=\ell$.
According to the stochastic growth processes defined in Subsection~\ref{secDESCsub1} we obtain the recurrence relation
\begin{align*}
 \P\{Y_{n+1,\ell,j}=m\} = \frac{c_1(m+\ell-2)+c_2}{c_1n+c_2}\P\{Y_{n,\ell,j}=m-1\}
+\frac{c_1(n+1-m-\ell)}{c_1n+c_2}\P\{Y_{n,\ell,j}=m\},
\end{align*}
for $m\ge 1$ and the initial value $Y_{j,\ell,j}=1$. The parameters $c_1,c_2$ occurring in this description of the law of $Y_{n,\ell,j}$ are determined by the fraction $\frac{c_2}{c_1}=\kappa$ as given in~\eqref{cc}.
Alternatively, $Y_{n,\ell,j}$ can be described as follows.
\begin{prop}
The r.v.\ $Y_{n,\ell,j}$ can be described by a sum of dependent random variables $A_{i,\ell,j}$, $i\ge j$, all taking values in $\{0,1\}$, and initial value $A_{j,\ell,j}=1$, where $A_{i,\ell,j}$ denotes the indicator variable of the event that label $i$ is a descendant of label $j$ conditioned on the event $K_{j}=\ell$, and we get
\begin{equation}
\label{RVDesc1}
Y_{n,\ell,j}= \sum_{i=j}^{n}A_{i,\ell,j},\quad \P\{A_{i+1,\ell,j}=1 \mid Y_{i,\ell,j}\}= \frac{c_1(\ell-1+Y_{i,\ell,j})+c_2}{c_1 i+c_2}.
\end{equation}
\end{prop}
\begin{proof}
It suffices to show that the probabilities $\P\{A_{i+1,\ell,j}=1 \mid Y_{i,\ell,j}\}$ are indeed given as stated above.
Given $K_j$ assume that node $v$ containing label $j$ has $m\ge 1$ descendants at time $i\ge j$. 
If $m< b+1-K_j$ then node $v$ has out-degree zero and capacity $c(v)=K_j+m-1$. It attracts a new label $i+1$
with probability $p(v)=p_{i+1}(v)$ determined directly by the stochastic growth rules:
\[
p(v)=
\begin{cases}
      \displaystyle{\frac{K_j+m-1}{i}}, & \quad  \text{bucket recursive trees}, \\[3mm]
      \displaystyle{\frac{(d-1)(K_j+m-1)+1}{(d-1)i+1}}, & \quad \dit, \\[3mm]
      \displaystyle{\frac{(\alpha+1)(K_j+m-1)-1}{(\alpha+1)i-1}}, & \quad \port.
\end{cases}
\]
Otherwise, assume that the number of descendants is given by $m\ge b+1-K_j$. 
Thus, node $v$ is saturated and the remaining $m-(b+1-K_j)$ labels are distributed amongst the $r$ non-root nodes $u_1,\dots,u_r$ in the subtree rooted at $v$. The probability $\P\{i+1 <_d j\}$ that label $i+1$ is a descendant of label $j$ is given by
\[
\P\{i+1 <_d j\}= \P\{i+1 <_c v\}+ \sum_{k=1}^{r}\P\{i+1 <_c u_k\},
\]
where $\P\{i+1 <_c x\}$ denotes the probability that label $i+1$ is contained in the node $x$. 
We note that 
\[
\sum_{k=1}^{r}c(u_k)=m-(b+1-K_j),\qquad\grad^{+}(v)+\sum_{k=1}^{r}\grad^{+}(u_k)=r.
\]
For bucket recursive trees we obtain
\[
\P\{i+1 <_d j\}= \frac{b}i + \sum_{k=1}^{r}\frac{c(u_k)}{i}
= \frac{b+m-(b+1-K_j)}{i}=\frac{K_j+m-1}{i}.
\]
For \Dit{} we have
\begin{align*}
\P\{i+1 <_d j\} &= \frac{(d-1)c(v)+1-\grad^{+}(v)}{(d-1)i+1} + \sum_{k=1}^{r}\frac{(d-1)c(u_k)+1-\grad^{+}(u_k)}{(d-1)i+1}\\
&= \frac{(d-1)b+r+1 + (d-1)(m-(b+1-K_j))-r}{(d-1)i+1}=\frac{(d-1)(K_j+m-1)+1}{(d-1)i+1}.
\end{align*}
Finally, for \Port{} we get
\begin{align*}
\P\{i+1 <_d j\} &= \frac{\grad^{+}(v)+(\alpha+1)c(v)-1}{(\alpha+1)i-1} + \sum_{k=1}^{r}\frac{\grad^{+}(u_k)+(\alpha+1)c(u_k)-1}{(\alpha+1)i-1}\\
&= \frac{r+(\alpha+1)(b+m-(b+1-K_j))-r-1}{(\alpha+1)i-1}=\frac{(\alpha+1)(K_j+m-1)-1}{(\alpha+1)i-1}.
\end{align*}
Summarizing we obtain the probabilities
\begin{equation*}
  \P\{i+1 <_d j \mid Y_{i,\ell,j}=m\} = 
	\begin{cases}
      \displaystyle{\frac{\ell-1+m}{i}}, & \quad  \text{bucket recursive trees}, \\[3mm]
      \displaystyle{\frac{(d-1)(\ell-1+m)+1}{(d-1)i+1}}, & \quad \dit, \\[3mm]
      \displaystyle{\frac{(\alpha+1)(\ell-1+m)-1}{(\alpha+1)i-1}}, & \quad \port,
\end{cases}
\end{equation*}
thus leading to the decomposition stated in \eqref{RVDesc1}.
\end{proof}
We remark in passing that, alternatively, $Y_{n,\ell,j}$ can also be described via a generalized P\'olya urn model; see~\cite{KubPanMomSeq}.

\smallskip

It is a key observation that the distribution of $Y_{n,\ell,j}$ is identical to the distribution of a r.v.\ $D_{n,\ell,j}$, counting so-called generalized descendants $D_{n,\ell,j}$ in ordinary families of increasing trees, which has been introduced and studied in \cite{KubPanCombDesc}. Namely, the r.v.\ $D_{n,\ell,j}$ also admits a description as a sum of dependent r.v.\ equivalent to \eqref{RVDesc1}, see \cite[equations $(7)$-$(8)$]{KubPanCombDesc}, thus $Y_{n,\ell,j} \stackrel{(d)}{=} D_{n,\ell,j}$. Unconditioning immediately leads to the following result.
\begin{prop}[Decomposition of the number of descendants]
The random variable $Y_{n,j}$ counting the number of descendants of label $j$ in a bucket increasing tree of size $n$ for the families bucket recursive trees, \Dit, and \Port is related to the random variable $D_{n,\ell,j}$, counting generalized descendants with parameter $\ell$ for the corresponding ordinary increasing tree families as studied in \cite{KubPanCombDesc}, where the parameter $\ell$ is given by the random variable $K_j$ measuring the initial bucket size of label $j$:
\[
Y_{n,j} \law D_{n,K_j,j}.
\]
\end{prop}

The probability mass function of $D_{n,\ell,j}$ as well as limit laws have been obtained in~\cite{KubPanCombDesc} using lattice path counting arguments. Using them this easily leads to the following limiting distribution result of $Y_{n,j}$, also slightly refining the results of~\cite{BucketPanKu2009} for bucket recursive trees. We thus omit the details.
\begin{coroll}[Limit laws for the number of descendants]
For $n \to \infty$ and $j=j(n)$, the random variable $Y_{n,j}$ has the following limit laws, depending on the random initial bucket size $K_j$ or its limit $K$.
\begin{enumerate}
\item The region for $j$ fixed. The normalized random variable $Y_{n,j}$ is asymptotically Beta-distributed, $Y_{n,j}/n \claw \beta(K_j+\frac{c_{2}}{c_{1}},j-K_j)$.
                    
\item The region for small $j$: $j \to \infty$ such that $j = o(n)$.
        The normalized random variable $j Y_{n,j}/n$ is
        asymptotically Gamma-distributed, $j Y_{n,j}/n \claw \gamma(K+\frac{c_{2}}{c_{1}},1)$.
  
\item The central region for $j$: $j \to \infty$ such that $j \sim \rho n$, with $0 < \rho < 1$.
        The shifted random variable $Y_{n,j}-1$ is asymptotically
        negative binomial-distributed, $Y_{n,j}-1\claw \NegBin(K+\frac{c_{2}}{c_{1}},\rho)$.
        
\item The region for large $j$: $j \to \infty$ such that $k := n-j = o(n)$.
        The random variable $Y_{n,j}$ has asymptotically all its mass concentrated at $1$, $
        \P\{Y_{n,j}=1\}=1+\mathcal{O}(\frac{k}{n})$.
\end{enumerate}
\end{coroll}
\begin{remark}
We mention the possibility to improve the distributional convergence of $Y_{n,j}$ in several ways. First, one can prove moment convergence. Second, local limit theorems can be deduced using the explicit expressions for the probability mass functions of $K_j$ and $D_{n,\ell,j}$. Third,
for fixed $j$ one can measure the difference between $Y_{n,j}/n$ and the limiting beta random variable in terms of a so-called martingale tail sum using discrete martingales. 
\end{remark}


\subsection{Node Degrees in bucket increasing trees} 
Let $X_{n,j}$ denote the random variable counting the out-degree of the bucket containing label $j$ in a size $n$ random bucket increasing tree.
By definition, we can decompose $X_{n,j}$ into a sequence of dependent indicator variables
\[
X_{n,j}=\sum_{\ell=j+1}^{n}\ind\{\ell <_a j\},
\]
where $\{\ell <_a j\}$ denotes the event that label $\ell$ is attached as a new node to the node containing label $j$. 
Similar to the case $b=1$, see for example~\cite{KubPan2005NodeDeg}, the random variable $X_{n,j}$ does not obey a uniform behavior, but it is similar to the r.v.\ number of descendants considered before. 
Basically, we will observe that for $b>1$ the random variable $X_{n,j}$ behaves similar
to the case $b=1$ once the bucket containing the label $j$ is fully saturated. 
Compared to the two ``stages'' of $Y_{n,j}$, insertion of label $j$ into a node of size $K_j$ and then attraction of new labels, we have three stages:
first, label $j$ is inserted into a node $v$ of size $K_j$; second, node $v$ attracts new labels until it is fully saturated; then, until time $n$ the node $v$ attracts new labels 
and its out-degree increases. 
In order to state the precise decompositions of $X_{n,j}$ we utilize our previous results for $K_j$ and $Y_{n,j}$. We also collect known results about additional random variables.

\bigskip

We introduce first the stopping time $\tau_{n,j}:=\min_{j\le\ell\le n}\{Y_{\ell,j}=b+1-K_j\}$ until the node containing label $j$ is saturated.
Let $\nabla$ denote the backward difference operator, i.e., $\nabla x_{k} := x_{k} - x_{k-1}$. The random variable $\tau_{n,j}$ can be expressed in terms of indicator variables as follows:
\[
\tau_{n,j}=\sum_{k=j}^{n}\ind\{Y_{k,j}= b+1-K_j,\,\nabla Y_{k,j}=1\}\cdot k + \ind\{Y_{n,j}< b+1-K_j\} \cdot n,
\]
where $Y_{j-1,j}=0$. Let $\P\{m<_t v \mid c(v)=b-1\}$ denote the conditional probability that node $v$, containing label $j$, attracts label $m$ conditioned on the event $c(v)=b-1$, as determined by the stochastic growth processes of the three tree families considered:
\[
\P\{m<_t v\mid c(v)=b-1\}=
\begin{cases}
      \displaystyle{\frac{b-1}{m-1}}, & \quad  \text{bucket recursive trees}, \\[3mm]
      \displaystyle{\frac{(d-1)(b-1)+1}{(d-1)(m-1)+1}}, & \quad \dit, \\[3mm]
      \displaystyle{\frac{(\alpha+1)(b-1)-1}{(\alpha+1)(m-1)-1}}, & \quad \port.
\end{cases}
\]
The probability mass function $\P\{\tau_{n,j}=m\}$ is readily obtained in terms of the probability mass functions of $K_j$, $Y_{n,j}$ and thus $Y_{n,\ell,j}$, where the constants $c_1,c_2$ are determined by~\eqref{cc} via $c_{2}/c_{1} = \kappa$. 

\smallskip

\begin{lemma}[Distribution of the saturation time]
For $1\le j \le b$ the random variable $\tau_{n,j}$ is deterministic: $\tau_{n,j}=b$.

For $j\ge b+1$ the probability mass function $\P\{\tau_{n,j}=m\}$ is given as follows:
 
\noindent{}For $m=j$ we have
\[
\P\{\tau_{n,j}=j\}=\P\{K_j=b\},
\]
for $j+1\le m<n$ we get
\[
\P\{\tau_{n,j}=m\}=\sum_{\ell=1}^{b-1}\P\{K_j=\ell\}\cdot\P\{Y_{m-1,\ell,j}=b-\ell\}\cdot \P\{m<_t v \mid c(v)=b-1\},
\]
whereas for $m=n$ we obtain
\[
\P\{\tau_{n,j}=n\}=\sum_{\ell=1}^{b-1}\P\{K_j=\ell\}\cdot \big( \P\{Y_{n,\ell,j}< b+1-\ell\}
+\P\{Y_{n-1,\ell,j}=b-\ell\}\cdot \P\{n<_t v\mid c(v)=b-1\}\big).
\]
\end{lemma}
\begin{remark}
In the case $j=b+1$ the expressions above simplify due to the fact that $K_{b+1}=1$, which is evident from the stochastic growth rules.
\end{remark}
\begin{proof}
For $1\le j\le b$ there is only one bucket present. It is saturated after the insertion of label $b$.
Assume that $j\ge b+1$. The probability $\P\{\tau_{n,j}=j\}$ is given by 
\[
\P\{Y_{j,j}= b+1-K_j,\,\nabla Y_{j,j}=1\}=\P\{1=b+1-K_j,\,Y_{j,j}-Y_{j-1,j}=1\}=\P\{K_j=b\}.
\]
For $j<m<n$ we obtain the probability of the event $\{Y_{k,j}= b+1-K_j,\,\nabla Y_{k,j}=1\}$
by conditioning on the initial bucket size $K_j$. Finally, for $j=n$ we also take into account the probability of the event 
$\{Y_{n,j}<b+1-K_j\}$.
\end{proof}

\smallskip

Let $\Be(p)$ denote a Bernoulli-distributed random variable:
\[
\P\{\Be(p)=1\}=p,\qquad \P\{\Be(p)=0\}=1-p.
\]
Furthermore, let $W_N(w_0,b_0)$ denote the number of white balls at time $N$ in a triangular urn model with balls of two colours and initial values $w_0>0$ and $b_0>0$, 
whose balanced replacement matrix given by
$
\left(
\begin{smallmatrix}
1& \alpha\\
0& 1+\alpha\\
\end{smallmatrix}
\right)
$.

We obtain the following result.
\begin{theorem}
The random variable $X_{n,j}$ counting the out-degree of label $j$ in a random bucket increasing tree of size $n$, with $1\le j\le n$, satisfies the following. 

\begin{itemize}
	\item For bucket recursive trees, $X_{n,j}$ is distributed as a random sum of mutually independent indicator variables:
\[
X_{n,j}\law\sum_{\ell=\tau_{n,j}+1}^{n}\ind(\ell<_c j),\qquad\text{with} \enspace \ind(\ell<_c j)=\Be(\textstyle{\frac{b}{\ell}}).
\]

\item For \Port, $X_{n,j}$ is distributed as the number of white balls in the balanced triangular urn model described above:
\[
X_{n,j}\law W_{n-\tau_{n,j}}(w_0,b_0),\quad w_0=b(\alpha+1)-1,\, b_0=(\alpha+1)(\tau_{n,j}-b).
\]
In both cases, the stopping time $\tau_{n,j}:=\min_{j\le\ell\le n}\{Y_{\ell,j}=b+1-K_j\}$ depends on $K_j$ as well as on $Y_{n,j}$.
\end{itemize}
\end{theorem}

\begin{remark}
A similar result holds for \Dit; however, the random variable $X_{n,j}$ is by definition bounded, thus we leave the result to the interested reader. 
We further note that the probability mass functions of $X_{n,j}$ can alternatively be obtained by a generating functions approach, but without revealing
the structural decompositions (compare with the corresponding results of~\cite{BucketPanKu2009} for bucket recursive trees).
\end{remark}

\section{Conclusion and Acknowledgments}
In this work we introduced two new families of bucket increasing trees, which can be generated by a stochastic growth process. 
We introduced a clustering process $\mathcal{C}$ for ordinary increasing trees that generate bucket increasing trees. 
Moreover, by modifying the map $\mathcal{C}$ we obtain bijections between certain
ordinary increasing tree families and families of bilabelled bucket increasing trees.
Additionally, we obtained a bijection between increasing diamonds and bilabelled bucket increasing trees. 

We analyzed structural properties of bucket increasing trees, in particular, the tree parameter $K_n$,
counting the initial bucket size of the node containing label $n$ in a random tree of size $n$. Using the combinatorial description as well as the tree evolution process, a study of further quantities in bucket increasing tree families is possible, e.g., we want to mention node distances.
Moreover, there exist relations of bucket increasing trees to further combinatorial structures as, e.g., certain models of series-parallel networks, see~\cite{KubPan2019}.

  \medskip

The authors warmly thank Henning Sulzbach for his feedback, improving the presentation of this work.

\end{document}